\documentclass[10pt,a4paper]{article}
\usepackage[a4paper]{geometry}
\usepackage[T1]{fontenc}
\usepackage[utf8]{inputenc}
\usepackage{authblk}
\usepackage[normalem]{ulem}
\RequirePackage[numbers]{natbib}
\usepackage{amsmath,amsfonts,amssymb,amsthm}
\usepackage{titlesec}
\titleformat{\section}
  {\normalfont\fontsize{11}{11}\bfseries}{\thesection}{1em}{}
\titleformat{\subsection}
  {\normalfont\fontsize{11}{11}\bfseries}{\thesubsection}{1em}{}
\titleformat{\subsubsection}
  {\normalfont\fontsize{11}{11}\bfseries}{\thesubsubsection}{1em}{}
\titleformat{\paragraph}
  {\normalfont\fontsize{11}{11}\bfseries}{\theparagraph}{1em}{}

\makeatletter
\newcommand{\pushright}[1]{\ifmeasuring@#1\else\omit\hfill$\displaystyle#1$\fi\ignorespaces}
\newcommand{\pushleft}[1]{\ifmeasuring@#1\else\omit$\displaystyle#1$\hfill\fi\ignorespaces}
\makeatother

\usepackage{natbib}
\usepackage[usenames,dvipsnames,svgnames,table]{xcolor}
\usepackage{mathtools}	
\usepackage{booktabs}
\usepackage{subfigure}
\usepackage{graphicx}
\usepackage{enumerate}
\usepackage[normalem]{ulem}
\mathtoolsset{showonlyrefs=true}

\numberwithin{table}{section}
\usepackage[hyperindex,breaklinks]{hyperref}
\hypersetup{frenchlinks=true}
\usepackage{float}

\makeatletter

\newtheorem{assumption}{Assumption}
\numberwithin{assumption}{section}
\newtheorem{theorem}{Theorem}
\numberwithin{theorem}{section}
\newtheorem{lemma}{Lemma}
\numberwithin{lemma}{section}
\newtheorem{proposition}{Proposition}
\numberwithin{proposition}{section}

\numberwithin{corollary}{section}

\numberwithin{definition}{section}
\newtheorem{example}{Example}
\numberwithin{example}{section}
\newtheorem{remark}{Remark}
\numberwithin{remark}{section}

\newcommand{\cC}{ {\mathcal{C}}}

\newcommand{\cF}{ {\mathcal{F}}}

\newcommand{\cK}{ {\mathcal{K}}}
\newcommand{\cL}{ {\mathcal{L}}}

\newcommand{\cO}{ {\mathcal{O}}}

\newcommand{\cW}{ {\mathcal{W}}}

\newcommand{\fW}{\mathfrak{W}}

\newcommand{\bX}{X}

\newcommand{\E}{ {\mathbb{E}}}

\newcommand{\K}{ {\mathbb{K}}}

\newcommand{\M}{\mathbb{M}}
\newcommand{\N}{\mathbb{N}}
\renewcommand{\P}{ {\mathbb{P}}}
\newcommand{\R}{\mathbb{R}}

\newcommand{\ind}{\mathbf{1}}
\newcommand{\dd}{ {\mathrm{d}}}

\newcommand{\MalWt}{\cW}

\renewcommand{\to}{\rightarrow}
\newcommand{\del}{\partial}

\newcommand{\x}{\times}

\DeclareMathOperator{\arcosh}{arcosh}

\newcommand{\eqlnostar}[2]{\begin{align}\label{#1}#2\end{align}}
\newcommand{\eqstar}[1]{\begin{align*}#1\end{align*}}
\newcommand{\eq}[1]{\ifthenelse{\equal{#1}{*}}
  {\eqstar}
  {\eqlnostar{#1}}
 }
\makeatother

\title{Branching diffusion representation of semi-linear elliptic PDEs and estimation using Monte Carlo method}

\author[1]{Ankush Agarwal*}
\affil[1]{Adam Smith Business School, University of Glasgow, University Avenue, G128QQ Glasgow, United Kingdom. Email:{\tt ankush.agarwal@glasgow.ac.uk} }
\author[2]{Julien Claisse}
\affil[2]{Centre de Math\'ematiques Appliqu\'ees (CMAP), \'Ecole Polytechnique and CNRS, Route de Saclay, 91128 Palaiseau Cedex, France. Email: {\tt julien.claisse@polytechnique.edu}}
\date{\today}

\begin{document}
\maketitle

\begin{abstract}
We study semi-linear elliptic PDEs with polynomial non-linearity and provide a probabilistic representation of their solution using branching diffusion processes. When the non-linearity involves the unknown function but not its derivatives, we extend previous results in the literature by showing that our probabilistic representation provides a solution to the PDE without assuming its existence. In the general case, we derive a new representation of the solution by using marked branching diffusion processes and automatic differentiation formulas to account for the non-linear gradient term. In both cases, we develop new theoretical tools to provide explicit sufficient conditions under which our probabilistic representations hold. As an application, we consider several examples including multi-dimensional semi-linear elliptic PDEs and estimate their solution by using the Monte Carlo method.



\end{abstract}

\section{Introduction}
In this paper, we are interested in the following class of semi-linear elliptic partial differential equations (PDEs): given a bounded domain $\cO \subset \R^d$, $f:\cO\times\R\times\R^d\to\R$, $h:\del\cO\to\R$,
\eqlnostar{eq:quasilinear}{
\cL u + f\bigl(u,Du\bigr) =0~~ \text{in } \cO, \quad~ u=h ~~ \text{on }\partial\cO,
}
where $\cL$ is the infinitesimal generator of a diffusion. This class of PDEs arises naturally in many fields of science~(see, \textit{e.g.}, \cite{badiale2011semilinear,lions1982semilinear} and references therein). In most cases, the related PDEs only involve non-linearity in the unknown function but not in its derivatives. For instance, the Poisson--Boltzmann equation or 
the steady states in nonlinear equations of the Klein--Gordon or Schr\"odinger type. Besides, the stationary Burgers equation or Hamilton-Jacobi-Bellman equations when the diffusion coefficient is uncontrolled are examples of PDEs where the non-linearity involves both the unknown function and its first order derivatives. In this paper, we provide existence results for this class of PDEs when the generator $f$ is polynomial and derive a new probabilistic representation of the solution which is well suited for numerical application, especially in high dimensions, using the Monte Carlo method.

The classical probabilistic approach for semi-linear PDEs relies on the theory of backward stochastic differential equations (BSDEs) initiated by Pardoux and Peng~\cite{pardoux1990bsde}. The case of elliptic PDE with Dirichlet condition was first studied by Darling and Pardoux~\cite{darling1997} for Lipschitz generator and later more general results were provided by Pardoux~\cite{pardoux1999} and Briand \textit{et al.}~\cite{briand2003}.  These results rely on the so-called monotonicity assumption, which requires $y\mapsto f(x,y,z) - \mu y$ to be decreasing for some constant $\mu\in\R$, and the assumption that $f$ has a linear or quadratic growth in $z$. We do not put such restrictions in our setting and propose a method to investigate a different class of elliptic PDEs with a generator which is a polynomial in $y$ and $z$.
  
Our probabilistic representation is based on \emph{branching diffusion processes}. These processes describe the evolution of a population of independent and identical particles moving according to a diffusion process. They were first introduced by Skorokhod~\cite{skorohod1964branching} and later studied, more thoroughly and systematically, in a series of papers by Ikeda \textit{et al.}~\cite{ikeda1969a,ikeda1969b,ikeda1969c}. In particular, these authors, using branching diffusion processes, established a probabilistic representation of semi-linear parabolic PDEs of the form 
\eqlnostar{eq:parabolicsemi-linear}{
\partial_t u + \cL u + \beta\left(\sum_{l\in\N}{p_l u^l} - u\right) =0 ~~ \text{in } [0,T)\times \R^d,\quad~ u(T,\cdot)=g ~~ \text{in } \R^d,
}
where $g:\R^d\to\R$, $\beta>0$ and $(p_l)_{l\in\N}$ is a probability mass function. More precisely, they consider a process where each particle moves according to a diffusion with generator $\cL$ and dies at an exponentially distributed random time with parameter $\beta$ to give birth to $l$ offsprings with probability $p_l$. Then, the solution of PDE~\eqref{eq:parabolicsemi-linear}  is expressed as $\E[\prod_{i=1}^{N_T} g(X_T^i)]$ where $N_T$ and $(X_T^i)_{i=1,\ldots,N_T}$ denote the number of particles and their positions at time $T$ respectively. The special case $p_2=1$, which corresponds to the celebrated Fisher--KPP equation, has been particularly well-studied (see, \textit{e.g.}, McKean~\cite{mckean1975kpp}).

In Henry-Labord\`ere~\textit{et al.}~\cite{henry2014numerical}, the authors extended the probabilistic representation of (path-dependent) PDE~\eqref{eq:parabolicsemi-linear} where the probability mass function $(p_l)_{l\in\N}$ is replaced by an arbitrary real-valued sequence of functions. See also Rasulov \textit{et al.}~\cite{rasulov2010}. This result was further extended in Henry-Labord\`ere \textit{et al.}~\cite{henry2016branching} to parabolic semi-linear PDE of the form
\eqlnostar{eq:parabolicquasilinear}{
\partial_t u + \cL u + f(u,Du) =0 ~~ \text{in } [0,T)\times \R^d,\quad~ u(T,\cdot)=g ~~ \text{in } \R^d,
}
where $f$ is a $(d+1)$--variate polynomial. 
The main idea consists of introducing marked particles which carry a weight function given by the Bismuth-Elworthy-Li formula to account for non-linearity in gradient of the solution.
Under conditions of ``small non-linearity'' or small maturity $T$, the authors showed that their probabilistic representation provides a continuously differentiable viscosity solution to PDE~\eqref{eq:parabolicquasilinear}. 
They also performed numerical simulations to illustrate the accuracy of their method to solve PDEs using the Monte Carlo method.

The aim of this paper is to extend the results in  Henry-Labord\`ere \textit{et al.} \cite{henry2014numerical,henry2016branching} to the case of elliptic PDEs with Dirichlet condition. The first result on the link between semi-linear elliptic PDE and branching diffusion processes was obtained by Watanabe~\cite{watanabe1965extinction} who derived a criterion for extinction of branching Brownian motion absorbed at the boundary of a domain.
Recently, Bossy \textit{et al.}~\cite{champagnat2015pde} extended this result to derive a probabilistic representation for PDE of the form
\eqlnostar{eq:champagnatpde}{
\cL u + \beta\left(\sum_{l\in\N}{c_l u^l} - u\right) =0~~ \text{in } \cO, \quad~ u=h ~~ \text{on }\partial\cO,
}
where $(c_l)_{l\in\N}$ is a sequence of real-valued functions. They used it to compute a Monte Carlo approximation of the solution to Poisson-Boltzmann equation in dimension three.
One of the critical assumptions in their work is the existence of a smooth solution to PDE~\eqref{eq:champagnatpde}. In this paper, we do not require this assumption as we show directly that the probabilistic representation provides a continuous viscosity solution to PDE~\eqref{eq:champagnatpde}.
The main difficulty compared to~\cite{henry2014numerical} is to ensure the continuity and integrability of the probabilistic representation. In particular, the arguments of small maturity used in~\cite{henry2014numerical} cannot be exploited here and instead, we resolve this problem by developing new tools which ensure that our probabilistic representation holds in small domain.

In the second part of our paper, we perform a rigorous analysis for the case of semi-linear elliptic PDE involving non-linearity in gradient of the unknown function. To the best of our knowledge, this is the first paper in the literature to provide a representation for this class of PDE using branching diffusion processes. In contrast with Henry-Labord\`ere \textit{et al.}~\cite{henry2016branching}, Malliavin calculus cannot be used in this setting since the exit time of a diffusion from a domain is not differentiable in the Malliavin sense. However, Delarue~\cite{delarue2003estimates} and Gobet~\cite{gobet2004revisiting} have established a suitable automatic differentiation formula, based on the work of Thalmaier~\cite{thalmaier1997adf}, which allows us to derive a probabilistic representation analogous to~\cite{henry2016branching}. Under the assumption that the non-linear gradient term vanishes at the boundary, it provides a continuously differentiable solution to PDE~\eqref{eq:quasilinear} when $f$ is a multivariate polynomial.

Importantly, our probabilistic representation provides a means to evaluate solutions of semi-linear elliptic PDEs by using the Monte Carlo method whose accuracy relies on the dimensionless central limit theorem. For this reason, it is particularly well suited for numerical applications in high dimension where deterministic methods usually fail to provide accurate estimates. Regarding the probabilistic approach, although different numerical methods for BSDEs have been introduced (see, \textit{e.g.},~\cite{bouchard2004numericbsde, zhang2004numericbsde}), to the best of our knowledge, they have never been used in the literature for elliptic PDEs.
\footnote{A related problem has been studied by Bouchard and Menozzi~\cite{bouchard2009} for the case of parabolic PDEs in a cylindrical domain with finite time horizon.}
In addition, the classical approach to solve BSDEs induces a discretization bias and involves a sequence of conditional expectation estimators. The available methods to estimate conditional expectations, such as least-squares regression (see Gobet {\em et al}.~\cite{gobet2005regression}), induce an additional approximation bias and are computationally expensive in high dimension.

We consider several numerical examples, including multi-dimensional elliptic PDEs, to illustrate the accuracy of our method. One of the critical steps in the algorithm consists of simulating the first exit time and position of a diffusion from a domain. There are three classical approaches to perform this task -- walk on spheres scheme, walk on parallelepipeds scheme and Euler discretization method (see Bouchard \textit{et al.}~\citep{bouchard2013exittime} and references therein).
For instance, the walk on spheres scheme was used by Bossy \textit{et al.}~\citep{champagnat2015pde} to estimate the solution of Poisson-Boltzmann equation in dimension three. In this paper, we consider numerical examples which involve Brownian motion over a rectangular domain and therefore, we generate unbiased samples for exit time and position using the walk on squares scheme as introduced in Faure~\cite{faure1992simulation} and in Milstein and Tretyakov~\cite{milstein1999simulation}, and implemented in the library developed by Lejay~\cite{lejay:inria-00561409}.

The rest of the paper is organized as follows. In the next section, we provide a precise formulation of the problem and introduce branching diffusion processes used to derive our probabilistic representation. Then, we consider the case of semi-linear PDE with linear gradient term in Section~\ref{sec:semilinear} and give explicit sufficient conditions 
to ensure that the probabilistic representation provides a continuous viscosity solution to the PDE. In Section~\ref{sec:quasilinear}, we perform the analysis  for semi-linear PDE with non-linear gradient term. In particular, we establish appropriate automatic differentiation formulas. Finally, in Section~\ref{section: examples}, we present several numerical examples to illustrate the applicability of our results in different settings. We provide the proof of the automatic differentiation formula in Appendix~\ref{proof_autodiff}.

\subsection{Notations}
Any element $x \in \R^d, d \geq 1,$ is a column vector with $i$th component $x_i$ and Euclidean norm $|x|.$ $x \cdot y$ denotes the usual dot product and $d(x,y)$ denotes the Euclidean distance for any $x,y \in\R^d.$  Given an open set $\cO \subset \R^d,$ $\bar{\cO}$ denotes its closure, $\del \cO$ its boundary and  $\mathrm{diam}(\cO)$ its diameter. For any  $g:\cO\to\R$, the supremum norm is defined as $\|g\|_{\infty} = \sup \{ |g(x)|: x \in \cO\}.$ $\cC(\cO)$ (resp. $\cC(\bar{\cO})$) denotes the set of continuous functions on $\cO$ (resp. $\bar{\cO}$). $\cC^k(\cO)$ (resp. $\cC^k(\bar{\cO})$) denotes the set of functions with continuous derivatives of all orders less than or equal to $k$ (resp. which further have continuous extensions to $\bar{\cO}$). $\cC^{k,\alpha}(\cO)$ (resp. $\cC^{k,\alpha}(\bar{\cO})$) are subspaces of $\cC^k(\cO)$ (resp. $\cC^k(\bar{\cO})$) consisting of functions whose partial derivatives of order less than or equal to $k$ are locally (resp. globally) H\"{o}lder continuous with exponent $ 0 < \alpha < 1.$ Given $T>0$, $\cC^{1,k}((0,T]\times\cO)$ denotes the set of functions with continuous first derivative with respect to the first component and continuous derivatives of all orders less than or equal to $ k$ with respect to the second component. For a smooth function $u(t,x),$ $Du$ and $D^2u$ stand for, respectively, its gradient (as a row vector) and Hessian matrix with respect to its second component. Furthermore, if $g: \R^m \to \R^n $ is a differentiable function, its gradient $D g = (\partial_{x_{1}}g(x), \ldots,\partial_{x_{m}}g(x) )$ takes values in $\R^n\otimes\R^m.$ For $d \geq 1,  \M^d$ denotes the set of all $d\times d$ matrices and $I_d \in \M^d$ the identity matrix. 

\section{Problem Definition}
\label{sec:framework}

\subsection{A Class of Semi-Linear Elliptic PDE}
Let $(\mu,\sigma):\R^d \times \R^d \to \R^d \times \M^d$ denote the drift and diffusion coefficient.
Then for a non-negative integer $m,$ we consider a subset $L \subset \N^{m+1}$ and a generator function $f:\R^d \times \R \times \R^d\to\R$ defined as 
\eqstar{
f(x,y,z) := \sum_{l = (l_0, l_1, \ldots, l_m) \in L} c_l(x) y^{l_0} \prod^m_{i=1} (b_i(x)\cdot z)^{l_i},
}
where $(c_l)_{l \in L}, c_l:\R^d \to \R$ and $(b_i)_{i =1,\ldots,m}, b_i: \R^d \to \R^d$ are sequences of functions. Furthermore, for every $l = (l_0, l_1, \ldots, l_m),$ we denote $\vert l \vert := \sum^m_{i=0}l_i.$ Given a bounded domain $\cO\subset \R^d$, we consider the following semi-linear elliptic PDE:
\eqlnostar{eq:main elliptic pde}{
\cL u + \beta \left(f\bigl(u,Du\bigr)-u\right) =0~~ \text{in } \cO,\quad~ u &= h~~ \text{on }\del \cO,
}
where $\beta$ is a positive constant, $h:\del\cO\to\R$ is the Dirichlet boundary condition and $\cL$ is the infinitesimal generator associated to a diffusion process with parameters $(\mu,\sigma)$. Under a set of general assumptions, we provide a probabilistic representation of solution $u$ using the theory of branching diffusion processes. We list the assumptions on parameters of PDE~\eqref{eq:main elliptic pde} which are needed for our results at the outset.

\begin{assumption}
\label{ass:pde}
\emph{(i)} The functions $b_i$, $c_l$ are continuous on $\bar{\cO}$. 
\newline
\emph{(ii)} The function $h$ is continuous on $\del\cO$. 
\newline
\emph{(iii)} The function $(x,y,\zeta)\mapsto \sum_{l \in L} c_l(x) y^{l_0} \prod^m_{i=1} \zeta_i^{l_i}$ is continuous on $\bar{\cO}\x\R\x\R^{m}$.
\end{assumption}
Part~(iii) above implies that the multivariate power series $(y,\zeta)\mapsto \sum_{l \in L} c_l(x) y^{l_0} \prod^m_{i=1} \zeta_i^{l_i}$ has an infinite radius of convergence for all $x\in\bar{\cO}$. It is a simplifying assumption which avoids the use of localization arguments in the rest of the paper.

\subsection{Marked Branching Diffusion Processes with Absorption}
\label{section:branching diffusion recap}

A (age-dependent) marked branching diffusion process is characterized by diffusion parameters $(\mu,\sigma)$, a probability density function $\rho: \R_+\to \R_+$, and a probability mass function $(p_l)_{l \in L}.$ In this process, we start with one particle of mark $0$ at position $x \in \R^d$ which undergoes a diffusion with parameters $(\mu,\sigma)$ during its lifetime distributed according to $\rho.$ At the end of its lifetime (arrival time), the particle dies and gives rise to $|l|=\sum_{i=0}^{m} l_i$ offsprings with probability $p_l$, among which $l_0$ have mark 0, $l_1$ have mark 1, and so on. After their birth, each offspring performs the same but an independent branching diffusion process as the parent particle. 
Additionally, we consider that particles are absorbed at the boundary, \textit{i.e.}, they die without giving rise to any offspring when they leave the domain $\cO$. 
In order to construct the above process, we denote by $\K:=\{\emptyset\}\cup\bigcup_{n\geq 1}{{\N}^n}$ the set of labels and we consider the probability space $(\Omega,\cF,\P)$ equipped with
\begin{itemize}
\item a sequence of i.i.d.\ positive random variables $(\tau^{k})_{k\in\K}$ distributed with density function $\rho$
\item a sequence of i.i.d.\ random elements $(I^k)_{k\in\K}$ with $\P(I^k = l) = p_l, l \in L$
\item a sequence of independent $d$-dimensional Brownian motions  $(W^{k})_{k\in\K}$ 
\end{itemize}
Furthermore, we consider the sequences $(\tau^{k})_{k\in\K}, (I^{k})_{k\in\K}$ and $ (W^{k})_{k\in\K}$ to be mutually independent. The age-dependent branching process is constructed as follows:

\begin{enumerate}
\item Start from one particle at position $x \in \R^d$ and index it by label $\emptyset.$ This is the common ancestor of all particles, the only particle of generation 0. For notational convenience, we simply write $\tau$, $I$ and $W$ instead of $\tau^{\emptyset}$, $I^{\emptyset}$ and $W^{\emptyset}$. The dynamic of particle $\emptyset$ is given as follows:
\begin{itemize}
\item The position $X^{\emptyset}=X^x$ of the particle during its lifetime is given by
\eqstar{
X^x_t = x + \int^t_0 \mu(X^x_s) \,ds + \int^t_0 \sigma(X^x_s) \,dW_s, \quad \P-\text{a.s.},
}
\item The arrival time of the particle is given by 
\eqstar{
T^{\emptyset}:=\tau\wedge\eta^x\quad\text{where}~~\eta^x := \inf \Big\{t \geq 0 ;\, X^x_t \notin \cO\Big\}.
}
\item At the arrival time, if $\eta^x\leq\tau$, then the particle $\emptyset$ dies without giving rise to any offspring, else the particle $\emptyset$ dies and gives rise to $\vert I\vert$ offsprings which belong to the first generation, and are indexed by label $i$ for $i = 0,\ldots,\vert I \vert-1.$

\item Given $I = (I_0, I_1, \ldots, I_m),$ we have $\vert I \vert = \sum^m_{i=0} I_i$ offspring particles, among which the first $I_0$ have mark 0, $I_1$ have mark 1, and so on, so that each particle has mark $i$ for $i = 0, \ldots,m.$
\end{itemize}  

\item For generation $n\geq 1,$ let the label for a particle be given as $k = (k_1,\ldots,k_{n-1},k_n) \in {\N}^n.$ Furthermore, denote by $k^- := (k_1,\ldots,k_{n-2},k_{n-1})$ the parent particle of $k$. The particle $k$ starts from $X^{k^-}_{T_{k^-}}$ at time $T_{k^-}$:
\begin{itemize}
\item The position $X^k$ of the particle during its lifetime is given by
\eqstar{
X^k_t = X^{k^-}_{T_{k^-}} + \int^t_{T_{k^-}} \mu(X^k_s) \,ds + \int^t_{T_{k^-}} \sigma(X^k_s) \,dW^k_{s-T_{k-}}, \quad \P-\text{a.s.}
}

\item The arrival time of the particle is given by
\eqstar{
T_k := \left(T_{k^-} + \tau^k\right)\wedge 
\inf \Big\{t \geq T_{k^-} ;\, X^k_t \notin \cO\Big\}.
}
\item At the arrival time, if $X^k_{T_k}\notin\cO$, then the particle $k$ dies without giving rise to any offspring, else the particle $k$ dies and gives rise to $\vert I^k\vert$ offsprings which belong to the $(n+1)$th generation, and are indexed by label $(k_1,\ldots,k_{n-1},k_n, i)$ for $i = 0,\ldots,\vert I^k \vert -1.$

\item Given $I^k = (I^k_0, I^k_1, \ldots, I^k_m),$ we have $\vert I^k \vert = \sum^m_{i=0} I^k_i$ offspring particles, among which the first $I^k_0$ have mark 0, $I^k_1$ have mark 1, and so on, so that each particle has mark $i$ for $i = 0, \ldots,m.$
\end{itemize}
\end{enumerate}
In addition, we denote by $\cK^x_n$ the collection of particles of the $n$th generation and by $\cK^x=\bigcup_{n\in\N} {\cK^x_n}$ the collection of all particles.
Finally, we introduce the filtration
\eqstar{ 
\cF_n := \sigma\left( \tau^{k}, I^{k}, W^{k}, k\in\{\emptyset\} \cup \bigcup_{i=1}^{n} {\N}^i \right),\quad n\in\N.
}

We make the following assumptions to ensure that the branching diffusion process above is well-defined and for further developments.
\begin{assumption}
  \label{ass:bdp}
  \emph{(i)} The probability distribution $(p_l)_{l\in L}$ satisfies $p_l>0$ and $\sum_{l\in L} { |l|  p_l} < \infty$.\\
  \emph{(ii)} The probability density $\rho$ is strictly positive.\\
  \emph{(iii)} The coefficients $(\mu,\sigma)$ are Lipschitz on $\bar{\cO}$.
\end{assumption}
Part (i) of the assumption above ensures that the number of particles remains finite in finite time in the underlying branching process, see, \textit{e.g.}, Athreya and Ney~\citep[Theorem 4.1.1]{athreya2012branching}. In addition, under assertion~(iii) , there exists a unique solution (up to the boundary) to the stochastic differential equation (SDE) corresponding to $(\mu,\sigma)$ and so the branching diffusion process is well-defined.

\begin{assumption}
  \label{ass:extinction}
  The branching diffusion process goes extinct almost surely.
\end{assumption}

It is clearly sufficient to assume that $\sum_{l\in L} { |l|  p_l} \leq 1$ for Assumption~\ref{ass:extinction} to hold.
However, since particles are absorbed at the boundary, we can derive much weaker conditions, especially if the domain is small.
For instance, in the case of branching Brownian motion with exponential lifetime of parameter $\beta$, Assumption~\ref{ass:extinction} is equivalent to 
\eqlnostar{eq:extinction}{
\beta\left(\sum_{l\in L} { |l|  p_l} - 1\right) - \frac{\lambda_1}{2} \leq 0,
} 
where $\lambda_1$ is the first positive eigenvalue of the Laplacian operator in the domain $\cO$, see Sevast$^\prime$yanov~\cite{sevastyanov1961extinction} or Watanabe~\citep{watanabe1965extinction}. Also see Remark~\ref{rem:extinction} below for further developments.

\section{Semi-Linear PDEs with Linear Gradient Term}
\label{sec:semilinear}

In this section, we are concerned with semi-linear PDEs with polynomial non-linearity involving the unknown function but not its derivatives, \textit{i.e.}, we assume that $m=0$ in Section~\ref{sec:framework} so that all  the particles have the same mark $0$ and  PDE~\eqref{eq:main elliptic pde} reads as 
\eqlnostar{eq:semi-linearpde}{
\cL u + \beta\left(f(u) - u\right) =0 ~~ \text{in } \cO, \quad~ u = h ~~ \text{on }\del \cO,
}
where $f(x,y)=\sum_{l\in L}{c_l(x) y^l}$, $L\subset\N$. Throughout this section, we suppose that Assumption \ref{ass:pde}--\ref{ass:extinction} remain valid. Additionally, we set $\rho$ as the density of exponential distribution with parameter $\beta$.

\subsection{Probabilistic Representation}
\label{sec:mainresultsemi}
We consider a branching diffusion process starting from $x\in\cO$ as in Section~\ref{section:branching diffusion recap} and we introduce the following random variable:
\eqstar{ 
\psi^x := \prod_{\substack{{k \in \cK^x}\\{X^k_{T_k}\notin \cO}}} h(X^k_{T_k}) \prod_{\substack{{k \in \cK^x}\\{X^k_{T_k}\in \cO}}} \frac{c_{I^k}(X^k_{T_k})}{p_{I^k}}.
}

\begin{proposition}
\label{prop:semi-linearcase}
Suppose that PDE~\eqref{eq:semi-linearpde} has a solution $u\in\cC^2(\cO)\cap\cC(\bar{\cO})$ such that the sequence $(\psi^x_n)_{n\in\N}$ given by
\eqstar{
 \psi^x_n := \prod_{\substack{{k \in \cup^n_{i=0}\cK^x_i}\\{X^k_{T_k}\notin \cO}}} h(X^k_{T_k}) \prod_{\substack{{k \in \cup^n_{i=0}\cK^x_i}\\{X^k_{T_k}\in \cO}}} \frac{c_{I^k}(X^k_{T_k})}{p_{I^k}} \prod_{k \in \cK^x_{n+1}}u(X^k_{T_{k^-}}),
}
is uniformly integrable, then we have $u(x) = \E[\psi^x]$.
\end{proposition}

\begin{proof}
By applying It\^o's formula to $(e^{-\beta t}u(X^x_t))_{t\geq 0}$, we obtain the following Feynman-Kac representation:
\eqstar{
u(x) = \E\Bigl[e^{-\beta{\eta^x}} h(\bX^x_{\eta^x}) + \int^{\eta^x}_0  \beta e^{-\beta s} f(\cdot,u)(\bX^x_s) \,ds\Bigr].
}
See, \textit{e.g.}, Freidlin~\cite[Theorem 2.2.1]{freidlin1985pde} for a detailed proof. 
Then, we can write
\eqlnostar{eq:feynman-kac branching1}{
u(x) &= \E\left[ h(\bX^x_{\eta^x}) \ind_{\tau \geq {\eta^x}} + f(\cdot,u)(\bX^x_{\tau}) \ind_{\tau < {\eta^x}} \right]\\
&= \E\left[h(X^{x}_{\eta^x}) \ind_{\tau \geq {\eta^x}} + \frac{c_{I}(X^{x}_{\tau})}{p_{I}} u^{I}(X^{x}_{\tau}) \ind_{\tau < {\eta^x}} \right].
}
Since an empty product is equal to $1$ by convention, it follows that 
$u(x)=\E[\psi^x_0]$. Next, since each offspring has the same dynamic as the parent particle, we can repeat the above calculations to write for $k\in\cK^x_1,$
\eqlnostar{eq:iterate soln1}{
u(X^k_{T_{k^-}}) = \E\left[h(X^k_{T_k}) \ind_{X^k_{T_k}\notin \cO} + \frac{c_{I^k}(X^k_{T_k})}{p_{I^k}} u^{I^k}(X^k_{T_k}) \ind_{X^k_{T_k} \in \cO} \,\bigg|\, \cF_0 \right].
}
We use the result in \eqref{eq:iterate soln1} and plug it back in \eqref{eq:feynman-kac branching1} to obtain $u(x)=\E[\psi^x_1]$ by conditional independence of particles in $\cK^x_1$ given $\cF_0$.
Similarly, we can show by iteration that for any $n\in\N,$ we have $u(x)=\E[\psi^x_n]$.
To conclude, it remains to observe that $\psi_n^x$ converges to $\psi^x$ almost surely in view of Assumption~\ref{ass:extinction}. Thus, if we suppose that $(\psi^x_n)_{n\in\N}$ is uniformly integrable, as $n \to \infty,$ we get $u(x)=\E[\psi^x]$.
\end{proof}

Proposition~\ref{prop:semi-linearcase} shows that there is at most one solution to PDE~\eqref{eq:semi-linearpde} satisfying an appropriate integrability condition and such a solution admits a probabilistic representation using branching diffusion processes. It is not completely satisfactory since one needs to prove first the existence of a classical solution and even then the uniform integrability condition is hard to derive as illustrated by the example of Section~\ref{sec:cosh}. 
In order to overcome these limitations, we next establish a result which shows directly that the probabilistic representation provides a (viscosity) solution of PDE~\eqref{eq:semi-linearpde}. This is the main result of this section. It is stated under abstract assumptions for which we will provide explicit sufficient conditions in the subsequent sections.

\begin{theorem}
\label{thm:semi-linear}
 Suppose $u:x\mapsto \E[\psi^x]$ is well-defined and  continuous on $\bar{\cO}$. Then $u$ is a viscosity solution of PDE~\eqref{eq:semi-linearpde}.
\end{theorem}

\begin{proof}
We first observe that by definition of $u$, it holds
\eqstar{
u(x) 
 = \E\left[h(X^{x}_{\eta^x}) \ind_{\tau \geq {\eta^x}} + \frac{c_{I}(X^{x}_{\tau})}{p_{I}} \prod_{i=0}^{I-1} \psi^{X^{x}_{\tau}}_i \ind_{\tau < {\eta^x}} \right].
}
where 
\eqstar{
  \psi^{X^{x}_{\tau}}_i := \prod_{\substack{{k=(i,\ldots)\in\cK^x}\\{X^k_{T_k}\notin \cO}}} h(X^k_{T_k}) \prod_{\substack{{k=(i,\ldots)\in\cK^x}\\{X^k_{T_k}\in \cO}}} \frac{c_{I^k}(X^k_{T_k})}{p_{I^k}}.
}
Furthermore, the branching property, which says that each offspring starts the same but an independent branching diffusion as the parent particle, yields that conditioned on $X^{x}_{\tau}$ and $I$, $(\psi^{X^{x}_{\tau}}_0,\ldots, \psi^{X^{x}_{\tau}}_{I-1})$ are independent random variables identical in law to $\psi^{X_0}$ where $X_0$ is distributed as $X^{x}_{\tau}$ and independent of $\cF_n$ for all $n\in\N$. Through this argument, we deduce that
\eqstar{
  \E\Big[\prod_{i=0}^{I-1} \psi^{X^{x}_{\tau}}_i \,\Big|\, X^{x}_{\tau}, I  \Big] \ind_{\tau<\eta^x} = u^{I}\left(X^{x}_{\tau}\right) \ind_{\tau<\eta^x}.
}
Working backward along the lines of the proof of Proposition~\ref{prop:semi-linearcase}, we deduce that 
\eqlnostar{eq:representation}{
u(x) = \E\Bigl[e^{-\beta{\eta^x}} h(\bX^x_{\eta^x}) + \int^{\eta^x}_0 \beta e^{-\beta s} f(\cdot,u)(\bX^x_s) \,ds\Bigr].
}
Next, for any $\delta>0,$ it follows from Markov property that
\eqlnostar{eq:viscosity}{
u(x) = \E\Bigl[e^{-\beta({\eta^x}\wedge \delta)} u(\bX^x_{{\eta^x}\wedge \delta}) + \int^{{\eta^x}\wedge \delta}_0 \beta e^{-\beta s} f(\cdot,u)(\bX^x_s) \,ds\Bigr].
}
The fact that $u$ is a viscosity solution of PDE~\eqref{eq:semi-linearpde} now follows from classical arguments. 
For the sake of completeness, let us show that $u$ is a viscosity subsolution. Consider $\varphi\in\cC^{2}_b(\cO)$ such that $x\in\cO$ is a maximum point of $u-\varphi$ and $u(x)=\varphi(x)$. 
First, we observe by applying It\^o's formula that 
\eqstar{
\E\left[e^{-\beta({\eta^x}\wedge \delta)} \varphi(\bX^x_{{\eta^x}\wedge \delta})\right] = \varphi(x) +
 \E\left[\int^{{\eta^x}\wedge \delta}_0  {e^{-\beta s} \left(\cL\varphi  - \beta\varphi \right)(\bX^x_s) \,ds}\right].
}
Since $u(x)=\varphi(x)$ and $u\leq \varphi$ otherwise, we deduce by using~\eqref{eq:viscosity} that 
\eqstar{
\E\left[\frac{1}{\delta}\int^{{\eta^x}\wedge \delta}_0 e^{-\beta s} \left(\cL\varphi - \beta \varphi + \beta f(\cdot,u)\right)(\bX^x_s) \,ds\right]\geq 0.
}
Since $y\mapsto f(y,u(y))$ is continuous and bounded, it follows from the mean value theorem and the dominated convergence theorem that 
\begin{equation*}
  \cL \varphi(x) + \beta\left(f(x,\varphi(x)) - \varphi(x)\right) \geq 0.
\end{equation*}
Thus $u$ is a viscosity subsolution of PDE~\eqref{eq:semi-linearpde}. The fact that $u$ is a viscosity supersolution results from similar arguments.

\end{proof}

\begin{remark}
 It will turn out from the study of the next sections that the explicit conditions we provide to ensure that $x\mapsto \E[\psi^x]$ is a solution to PDE~\eqref{eq:semi-linearpde} also entail that $x\mapsto \E[|\psi^x|]$ is a solution to
\eqstar{
\cL v + \beta\left(\sum_{l\in L} \left| c_l\right| v^{l}-v\right) = 0 ~~ \text{in } \cO, \quad~ 
v = \left| h\right| ~~ \text{on }\del \cO.
}
In particular, we observe that the monotonicity assumption is not relevant in our setting. 
The limiting assumption in our approach is the integrability condition on $(\psi^x)_{x\in \cO}$ which ensures that $u:x\mapsto\E[\psi^x]$ is well-defined. See Section~\ref{sec:integrability} for more details.
\end{remark}

\begin{remark}
We can also work with a general lifetime distribution $\rho$ and extend the arguments above, as done in Section~\ref{sec:quasilinear}, to derive another probabilistic representation which is given as
\eqstar{ 
 u(x) = \E\left[\prod_{\substack{{k \in \cK^x}\\{X^k_{T_k}\notin \cO}}} \frac{e^{-\beta\Delta T_k} h(X^k_{T_k})}{\bar{F}(\Delta T_k)} \prod_{\substack{{k \in \cK^x}\\{X^k_{T_k}\in \cO}}} \frac{\beta e^{-\beta\Delta T_k} c_{I^k}(X^k_{T_k})}{p_{I^k} \rho(\Delta T_k)}\right],
}
where $\Delta T_k := T_k - T_{k^-}$ is the lifetime of particle $k$ and $\bar{F}(t):=\int_t^{\infty}{\rho(s)\,ds}$, $t\geq 0$.
However, this leads to more stringent assumptions when studying the integrability condition as in Section~\ref{sec:integrability}.
\end{remark}

\subsection{Continuity Assumption}
\label{sec:explicitsemi-linear}
\label{sec:continuity}

Let us give explicit sufficient conditions for the continuity assumption in Theorem~\ref{thm:semi-linear} to hold. We use two different approaches leading to slightly different conditions. Both approaches rely on appropriate integrability conditions on $(\psi^x)_{x\in\cO}$ which we discuss in the next section.

\subsubsection{PDE Approach}

\begin{assumption}\label{ass:continuity_sufficient}
\emph{(i)} The diffusion coefficient $\sigma$ is uniformly elliptic.\footnote{There exists $\lambda>0$ such that $\sigma\sigma^*(x)\geq \lambda I_d$ for all $x\in\bar{\cO}$.}\\
\emph{(ii)}The boundary $\partial\cO$ is of class $\cC^{1,\alpha}$.
\end{assumption}

\begin{proposition}
\label{prop:continuity}
Suppose Assumption~\ref{ass:continuity_sufficient} holds and $(\psi^x)_{x\in\cO}$ is uniformly bounded in $L^1$, then the map $u:x\mapsto \E[\psi^x]$ is continuous.
\end{proposition}

The proof follows immediately from~\eqref{eq:representation} and the following lemma.

\begin{lemma}\label{lem:continuity}
 Suppose Assumption~\ref{ass:continuity_sufficient} holds, then the following statements are satisfied:\\
 \emph{(i)} The map $\bar{\cO}\ni x\mapsto\E[e^{-\beta{\eta^x}} h(\bX^x_{\eta^x})]$ is continuous.\\
 \emph{(ii)} For any $g:\cO\to\R$ bounded measurable, the map $\bar{\cO}\ni x\mapsto\E[\int_0^{\eta^x}{e^{-\beta s}g(\bX^x_s)\,ds}]$ is continuous.
\end{lemma}

\begin{proof}
(i) Let us first study the continuity of $x\mapsto\E[e^{-\beta{\eta^x}} h(\bX^x_{\eta^x})]$. Under the assumptions above, it is well-known that there exists a smooth solution $\varphi\in\cC^2(\cO)\cap\cC(\bar{\cO})$ to the following PDE:
\eqlnostar{eq:linearpde}{
\cL \varphi - \beta \varphi = 0 ~~ \text{in } \cO, \quad~ \varphi = h ~~ \text{on } \del\cO.
}
See, \textit{e.g.}, Gilbarg and Trudinger~\cite[Theorem 6.13]{gilbarg2015elliptic}. By It\^o's formula, we deduce that $\varphi(x)=\E[e^{-\beta{\eta^x}}$ $h(\bX^x_{\eta^x})]$ and the conclusion follows. 

\noindent (ii)Let us now turn to the continuity of  $x\mapsto\E[\int_0^{\eta^x}{e^{-\beta s} g(\bX^x_s) \,ds}]$. We start by showing that $x\mapsto\E[g(\bX^x_s)\ind_{s<{\eta^x}}]$ is continuous for $s>0$. If $g$ is continuous and $g=0$ on $\del\cO$, it is known that the unique smooth solution $\chi\in \cC^{1,2}((0,T]\times \cO)\cap \cC([0,T]\times \bar{\cO})$ of 
\eqlnostar{eq:parabolic}{
\partial_t \chi - \cL \chi &=0,\quad~ \text{in } (0,T]\times\cO, \nonumber\\
\chi(0,\cdot) &= g,\quad~ \text{on } \cO, \nonumber\\
\chi &= 0, \quad~ \text{on } (0,T]\times \del\cO,
}
is of the form
\eqstar{
\chi(s,x) = \int_{\cO} {G(s,x;0,y) g(y) \,dy}, \quad 0 < s \leq T,
}
where $G$ is the so-called Green function of PDE~\eqref{eq:parabolic} (see, \textit{e.g.}, Lady\v{z}enskaja \textit{et al.}~\cite[Theorem 4.16.2]{ladyzenskaja1968parabolic}). Then, it follows from It\^o's formula that 
\eqstar{
 \E[g(\bX^{x}_s)\ind_{s<{\eta^x}}] = \int_{\cO} {G(s,x;0,y) g(y) \,dy}.
}
In particular, $y\mapsto G(s,x;0,y)$ appears as the density of $\bX^{x}_s$ on the event $\{s< {\eta^x}\}$, and so the identity above remains valid for any $g$ bounded measurable. Furthermore, $x\mapsto G(s,x;0,y)$ is continuous and satisfies
\eqstar{
\left|G(s,x;0,y)\right| \leq C s^{-\frac{d}{2}} e^{-C\frac{\left|x-y\right|^2}{s}},
}
for some constant $C>0$ depending on $T$, see~\cite[Equation 4.16.16]{ladyzenskaja1968parabolic}. Hence the desired result follows from the dominated convergence theorem. To conclude, it remains to observe that 
\eqstar{
\E\left[\int_0^{\eta^x}{e^{-\beta s} g\left(\bX^x_s\right) \,ds}\right] = \int_0^{+\infty}{e^{-\beta s} \E\left[g\left(\bX^x_s\right)\mathbf{1}_{s<\eta^x}\right] \,ds}
}
and to apply once again the dominated convergence theorem.

\end{proof}

\subsubsection{Probabilistic Approach}

\begin{assumption}\label{ass:continuity_sufficient_proba}
\emph{(i)}The diffusion coefficient $\sigma$ is uniformly elliptic on $\del\cO$.\footnote{There exists $\lambda>0$ such that $\sigma\sigma^*(x)\geq \lambda I_d$ for all $x\in\del\cO$.}\\
\emph{(ii)}The boundary $\partial\cO$ satisfies an exterior cone condition.\footnote{See, \textit{e.g.}, Gilbarg and Trudinger~\cite[Problem 2.12]{gilbarg2015elliptic} for a definition.}\\
\emph{(iii)}The stopping time $\eta^x$ is finite almost surely.
\end{assumption}

Sufficient conditions for Part (iii) to hold are provided in Freidlin~\cite[Lemma 3.3.1]{freidlin1985pde}. For instance, it suffices to assume that there exists $1\leq i\leq d$ such that $\sum_{j=1}^d \sigma_{ij}^2(x)>0$ or $|\mu_i(x)|>0$ for all $x\in\cO$.

\begin{proposition}
\label{prop:continuity_proba}
Suppose Assumption~\ref{ass:continuity_sufficient_proba} holds and $(\psi^x)_{x\in\cO}$ is uniformly integrable, then $u:x\mapsto \E[\psi^x]$ is continuous.
\end{proposition}

\begin{proof}
Clearly it suffices to prove that $\bar{\cO}\ni x\mapsto \psi^x$ is almost surely continuous, in the sense that for all sequence $(x_n)_{n\in\N}$ converging to $x$, it holds
\eqstar{
 \P\left(\lim_{n\to\infty} \psi^{x_n} = \psi^x \right)=1.
}
This essentially follows from the almost sure continuity of $x\mapsto\eta^x$ stated in Lemma~\ref{lem:continuity_proba} below.

\noindent(i) Let us first show that the contribution of the first particle to $\psi^x$ is almost surely continuous, \textit{i.e.}, 
\eqlnostar{eq:first}{
 \lim_{n\to\infty} {h(X^{x_n}_{\eta^{x_n}}) \ind_{\tau \geq {\eta^{x_n}}}} &= h(X^x_{\eta^x}) \ind_{\tau \geq {\eta^x}}, \quad \P-\text{a.s.},\\
  \lim_{n\to\infty} {\frac{c_I(X^{x_n}_{\tau})}{p_I} \ind_{\tau < {\eta^{x_n}}}} &= \frac{c_I(X^x_{\tau})}{p_I}\ind_{\tau < {\eta^x}}, \quad\P-\text{a.s.}\label{eq:second}
}
To achieve this, let us consider the set
\eqstar{
  \Omega^{\emptyset} := \left\{\eta^{x} \neq \tau\right\}\cap \left\{ \lim_{n\to\infty} \eta^{x_n} = \eta^x\right\}\cap \left\{ \lim_{n\to\infty} X^{x_n} = X^x\right\}.
 }
In view of Lemma~\ref{lem:continuity_proba} below, it is clear that $\P(\Omega^{\emptyset})=1$. Furthermore, one easily checks that, for every $\omega\in\Omega^{\emptyset}$, both~\eqref{eq:first}--\eqref{eq:second} hold.

\noindent(ii) Let us show next that the contribution of the particles of the first generation to $\psi^x$ is almost surely continuous. For every $i\in\N$, we denote by $(X^{i,x}_s)_{s\geq 0}$ the unique solution of 
\eqstar{
X^{i,x}_t = x + \int^t_{0} \mu(X^{i,x}_s) \,ds + \int^t_{0} \sigma(X^{i,x}_s) \,dB^i_s, \quad \P-\text{a.s.},
}
where the Brownian motion $B^i$ is defined by 
\eqstar{
B^i_t := W_{t\wedge\tau} + W^i_{t-\tau} \ind_{t
\geq\tau}.
}
Clearly,  $(X^{i,x}_s)_{s\geq 0}$ coincides with the trajectory $X^i$ of particle $i$ during its lifetime. Let $\eta^{i,x}$ be the first exit time of $X^{i,x}$ from $\cO$ and consider
\eqstar{
  \Omega^{i} := \left\{\eta^{i,x} \neq \bar{\tau}^i\right\}\cap \left\{ \lim_{n\to\infty} \eta^{i,x_n} = \eta^{i,x}\right\}\cap \left\{ \lim_{n\to\infty} X^{i,x_n} = X^{i,x}\right\}.
 }
 where $\bar{\tau}^i:=\tau+\tau^i$. Once again, $\P(\Omega^i)=1$ and
\eqstar{
 \lim_{n\to\infty} {h\left(X^{i,x_n}_{\eta^{i,x_n}}\right) \ind_{\bar{\tau}^i \geq \eta^{i,x_n}}} &= h\left(X^{i,x}_{\eta^{i,x}}\right) \ind_{\bar{\tau}^i \geq {\eta^{i,x}}}, \quad\text{on }\Omega^i, \\
 \lim_{n\to\infty} {\frac{c_{I^i}\left(X^{i,x_n}_{\bar{\tau}^i}\right)}{p_{I^i}} \ind_{\bar{\tau}^i < {\eta^{i,x_n}}}} &= \frac{c_{I^i}\left(X^{i,x}_{\bar{\tau}^i}\right)}{p_{I^i}}\ind_{\bar{\tau}^i < \eta^{i,x}}, \quad\text{on }\Omega^i.
}
Thus, we conclude that on the set $\Omega^{\emptyset}\cap\bigcap_{i\in\N}{\Omega^i}$ of probability $1$, the contributions of the particles of generations $0$ and $1$ to $\psi^x$ is almost surely continuous.

\noindent (iii) The desired result follows by iteration. We denote for all $n\geq 2$, $k\in\N^n$,
\eqstar{
  \Omega^{k} := \left\{\eta^{k,x} \neq \bar{\tau}^k\right\}\cap \left\{ \lim_{n\to\infty} \eta^{k,x_n} = \eta^{k,x}\right\}\cap \left\{ \lim_{n\to\infty} X^{k,x_n} = X^{k,x}\right\},
}
where $\bar{\tau}^k:=\bar{\tau}^{k^-}+\tau$ and $(X^{k,x}_s)_{s\geq 0}$ is the unique solution of 
\eqstar{
X^{k,x}_t = x + \int^t_{0} \mu(X^{k,x}_s) \,ds + \int^t_{0} \sigma(X^{k,x}_s) \,dB^k_s, \quad \P-\text{a.s.},
}
with
\eqstar{
B^k_t := B^{k^-}_{t\wedge\bar{\tau}^{k^-}} +  W^k_{t - \bar{\tau}^{k^-}} \ind_{t\geq\bar{\tau}^{k^-}}.
}
Then on the set $\bigcap_{k\in\K}\Omega^{k}$ of probability $1$,  it holds that $\lim_{n\to\infty} \psi^{x_n}=\psi^x$.
\end{proof}

\begin{lemma}\label{lem:continuity_proba}
 Suppose Assumption~\ref{ass:continuity_sufficient_proba} holds, then the map $\bar{\cO}\ni x\mapsto \eta^x$ is almost surely continuous, in the sense that for all sequence $(x_n)_{n\in\N}$ converging to $x$, it holds
 \eqstar{
  \P\left(\lim_{n\to\infty} \eta^{x_n} = \eta^x \right)=1.
 }
\end{lemma}

\begin{proof}
The proof is due to Darling and Pardoux~\cite{darling1997}, see also \cite[Proposition~4.4]{pardoux1998bsde}. More precisely, the authors show that if the stopping time 
\eqstar{
\tau^x:=\left\{s\geq 0,\, X^x_s\notin\bar{\cO}\right\},
}
is finite almost surely and $\P\left(\tau^x > 0\right) = 0$ for all $x\in\del\cO$, then $x\mapsto \tau^x$ is almost surely continuous. 
It remains to observe that, under Assumption~\ref{ass:continuity_sufficient_proba} (i)--(ii), $\P(\tau^x=0)=1$ for all $x\in\del\cO$ (see Bass~\cite[Corollary 3.3.2]{bass98} or Pinsky~\cite[Theorem 2.3.3]{pinsky95}) and thus $\tau^x=\eta^x$ for all $x\in\bar{\cO}$.
\end{proof}

\subsection{Integrability Condition}
\label{sec:integrability}

In this section, we provide explicit sufficient conditions to verify the integrability conditions on $(\psi_x)_{x\in\cO}$ required for Propositions~\ref{prop:continuity} and \ref{prop:continuity_proba}. 
Actually we study boundedness of $(\psi_x)_{x\in\cO}$ in $L^q$ for $q\geq 1$. In particular, the case $q=2$ ensures that the corresponding Monte Carlo estimator has finite variance. 
Let us introduce the constant
\eqlnostar{eq:const integrability}{
C_0:= \max\left(\|h\|_{\infty}, \sup_{l\in L} {\left\{\frac{\|c_l\|_{\infty}}{ p_l}\right\}}\right).
}
Clearly, it holds $|\psi^x|\leq C_0^{|\cK^x|}$ where $|\cK^x|$ denotes the cardinality of the set $\cK^x$, \textit{i.e.}, the total number of particles. 
In particular, if $C_0\leq 1$, then $|\psi^x|\leq 1$.
To the best of our knowledge, this is the only condition that has been used so far in the literature to ensure the integrability of $\psi^x$. 

In the rest of this section, we provide more technical but weaker conditions. First, we establish the desired result under minimal assumptions.

\begin{proposition}
\label{prop:unifint}
Suppose that there exists a non-negative function $v\in\cC^2(\cO)\cap\cC(\bar{\cO})$ satisfying
\eqstar{
\cL v + \beta\left(\sum_{l\in L} \frac{\left| c_l\right|^{q}}{p_l^{q-1}} v^{l}-v\right) \leq 0 ~~ \text{in } \cO, \quad~ 
v \geq\left| h\right|^q ~~ \text{on }\del \cO,
}
then we have $\E[|\psi^x|^q]\leq v(x)$. In particular, $(\psi^x)_{x\in\cO}$ is uniformly bounded in $L^q$.
\end{proposition}

\begin{proof}
First, we observe that It\^o's formula yields
\eqstar{
v(x) \geq \E\Bigl[e^{-\beta{\eta^x}} \left| h(\bX^x_{\eta^x})\right|^q + \int^{\eta^x}_0 \beta e^{-\beta s}\sum_{l\in L} \frac{\left| c_l(\bX^x_s)\right|^q}{p_l^{q-1}} v^{l}(\bX^x_s) \,ds\Bigr].
}
Next, by repeating the arguments of Proposition~\ref{prop:semi-linearcase}, we get
\eqstar{
v(x) &\geq \E\left[\prod_{\substack{{k \in \cup^n_{i=0}\cK^x_i}\\{X^k_{T_k}\notin \cO}}}\left|h(X^k_{T_k})\right|^q \prod_{\substack{{k \in \cup^n_{i=0}\cK^x_i}\\{X^k_{T_k} \in \cO}}} \frac{\left|c_{I^k}(X^k_{T_k})\right|^q}{p_{I^k}^q} \prod_{k \in \cK^x_{n+1}}v(X^k_{T_{k^-}})\right].
}
Then the conclusion follows easily from Fatou's lemma.
\end{proof}

In practice, one can look for a constant supersolution in order to apply Proposition~\ref{prop:unifint}. For instance, for the case $q=1$, such a solution exists if and only if $\sum_{l\in L} {\left| c_l\right|} \|h\|_{\infty}^{l}\leq \|h\|_{\infty}.$
Otherwise, finding a supersolution has to be done on a case by case basis and might turn out to be a difficult task. For this reason, we provide an alternative result for which conditions are easier to verify than Proposition~\ref{prop:unifint} and less stringent than assuming $C_0\leq 1$. Essentially, it signifies that the integrability condition is satisfied if the domain is sufficiently small.

The idea is to dispose of the spatial parameter $x$ by introducing a branching process which stochastically dominates the branching diffusion process. Denote 
\eqstar{
\delta := 1-\inf_{x\in\cO}\left\{\E\left[e^{-\beta\eta^x}\right]\right\}.
}
Let us introduce a new probability mass function $(\tilde{p}_l)_{l\in \tilde{L}}$ with $\tilde{L}:=L\cup\{0\}$ as follows:
\eqstar{
\tilde{p}_0 := 1-\delta + \delta p_0 ~~ \text{and} ~~ \tilde{p}_l &:= \delta p_l ~~\text{for all }l\geq 1,
}
and the corresponding transition matrix $\tilde{P}=(\tilde{P}_{i,j})_{i,j\geq 0}$ given by
\eqstar{
\tilde{P}_{0,0} = 1 ~~ \text{and} ~~ \tilde{P}_{i,i+l-1} = \tilde{p}_l ~~ \text{for all } i\geq 1, l\in \tilde{L}.
}

\begin{proposition}
\label{lem:unifinteasy}
\label{prop:extinction}
Denote by $R$ the common radius of convergence of $f(s):=\sum_{l\in L} {p_l s^l}$ and $\tilde{f}(s):=\sum_{l\in \tilde{L}} {\tilde{p}_l s^l}$. 
If $R>1$ and $\sum_{l\in\tilde{L}}{l\tilde{p}_l}<1$, then 
$(\psi^x)_{x\in\cO}$ is uniformly bounded in $L^q$ provided that $ C^q_0\leq\gamma$ where
\eqlnostar{eq:limit threshold}{
\gamma := \frac{s^*}{\tilde{f}(s^*)}=\frac{s^*}{1-\delta + \delta f(s^*)},
}
where $s^*$ is the solution (if any) to $s\tilde{f}'(s)=\tilde{f}(s)$ and $s^*=R$ otherwise. In addition, if $R=\infty$, then $\gamma$ goes to infinity as $\mathrm{diam}(\cO)$ goes to zero.
\end{proposition}

\begin{proof}
(i) Denote by $(N^x_n)_{n\in\N}$ and $(\tilde{N}_n)_{n\in\N}$ the number of particles at the $n$th jump time (arrival time) in the branching diffusion process and in a branching process with offspring distribution $(\tilde{p}_l)_{l\in \tilde{L}}$ respectively. In particular, $(\tilde{N}_n)_{n\in\N}$ is a Markov chain with transition matrix $\tilde{P}$. Let us first show that $(N^x_n)_{n\in\N}$ is stochastically dominated by $(\tilde{N}_n)$, \textit{i.e.},  $\P(N^x_n\geq l)\leq \P(\tilde{N}_n\geq l)$ for all $n,l\geq 1$. We observe that 
\eqstar{
\tilde{p}_0 = \inf_{x\in\cO}\left\{\P\left(N^x_1=0\right)\right\} ~~ \text{ and } ~~ \tilde{p}_l = \sup_{x\in\cO}\left\{\P\left(N^x_1=l\right)\right\} ~~ \text{for }l\in\tilde{L}\setminus\{0\}.
}
It follows that  $\P(N^x_1\geq l)\leq \P(\tilde{N}_1\geq l)$ for all $l\geq 1$. 
For the incremental step, we first observe that for all $l\geq 1$, 
\eqstar{
\P\left(N^x_{n+1}\geq l, N^x_n\geq l\right) 
	& \leq \P\left(N^x_{n}\geq l\right) - \P\left(N^x_{n}=l, N^x_{n+1} = l - 1\right) \\
	& \leq \P\left(N^x_{n}\geq l\right) -  \tilde{p}_0 \P\left(N^x_{n}=l\right) \\
    & \leq (1-\tilde{p}_0)\P\left(N^x_{n}\geq l\right) +  \tilde{p}_0 \P\left(N^x_{n}\geq l+1\right),
}
where the second inequality follows from Markov property. 
In addition, it holds 
\eqstar{
\P\left(N^x_{n+1}\geq l, N^x_n\leq l-1 \right) 
	& \leq \sum_{j=1}^{l-1} {\P\left(N^x_n=j,\, N^x_{n+1}-N^x_n\geq l-j\right)} \\
	& \leq \sum_{j=1}^{l-1} {\sum_{i=l-j+1}^{\infty} {\tilde{p}_{i}} {\P\left(N^x_n=j\right)}} \\
    & \leq \sum_{i=l}^{\infty}\tilde{p}_i \P\left(N^x_n\geq 1\right) - \sum_{i=2}^{\infty} {\tilde{p}_i} \P\left(N^x_n\geq l\right) + \sum_{j=2}^{l-1} \tilde{p}_{l-j+1} \P\left(N^x_n\geq j\right),
}
where the second inequality follows once again from Markov property.
Combining both computations above, we deduce that
\eqstar{
\P\left(N^x_{n+1}\geq l\right) 
 \leq \sum_{i=l}^{\infty}\tilde{p}_i \P\left(N^x_n\geq 1\right) +  \sum_{j=2}^{l-1} \tilde{p}_{l-j+1}\P\left(N^x_n\geq j\right) + \tilde{p}_1 \P\left(N^x_n\geq l\right) +  \tilde{p}_0 \P\left(N^x_n \geq l+1\right).
}
Hence the desired result follows easily by induction and Markov property.

\noindent (ii) Let us show next that $\E[|\psi^x|]\leq \E[C_0^{\zeta}]$ where $\zeta$ the extinction time  of $(\tilde{N}_n)_{n\in\N}$. First we recall that $|\psi^x|\leq C_0^{|\cK^x|}$ where $|\cK^x|$ denotes the cardinality of the set $\cK^x$. We observe further that $|\cK^x|$ coincides with the extinction time of $(N^x_n)_{n\in\N}$. 
Clearly, Step~(i) above yields that $\P(|\cK^x|\geq n)\leq \P(\zeta\geq n)$ for all $n\geq 1$, and thus 
\eqstar{
\E[|\psi^x|]\leq \E[C_0^{|\cK^x|}]\leq \E[C_0^{\zeta}].
}

\noindent(iii) We are now in a position to conclude the proof. It follows from Daley~\cite[Theorem~2]{daley1969qsd} that the power series $\E[s^{\zeta}]$ converges on its radius of convergence $\gamma=\frac{s^*}{\tilde{f}(s^*)}$. 
In addition, if $R=\infty$, then there exists a solution $s^*$ to $s \tilde{f}'(s)=\tilde{f}(s)$. Using further $\tilde{f}(s)= 1-\delta+\delta f(s)$, we deduce that
\eqstar{
s^* f'(s^*) - f(s^*) = \frac{1-\delta}{\delta}.
}
It follows that $s^*$ goes to infinity as $\delta$ goes to zero, or equivalently, $\mathrm{diam}(\cO)$ goes to zero. Then we have 
\eqstar{
\gamma = \frac{1}{\tilde{f}'(s^*)} = \frac{1}{\delta f'(s^*)} = \frac{s^*}{1-\delta} \left(1-\frac{f(s^*)}{s^* f'(s^*)}\right). 
}
To conclude, it remains to observe that $\frac{f(s^*)}{s^* f'(s^*)}$ is bounded away from $1$ since $s\mapsto\frac{f(s)}{sf'(s)}$ is decreasing.

\end{proof}

\begin{remark}
\label{rem:extinction}
 It follows immediately from Step~(i) of the proof of Proposition~\ref{prop:extinction} that if $\sum_{l\in \tilde{L}} {l \tilde{p}_l} \leq 1$, then the branching diffusion process goes extinct almost surely, \textit{i.e.}, Assumption~\ref{ass:extinction} holds.
\end{remark}

\section{Semi-Linear PDEs with Non-Linear Gradient Term}
\label{sec:quasilinear}
In this section we study the case of semi-linear PDE with non-linearity in gradient of the solution, \textit{i.e.}, we assume that $m\geq 1$ in Section~\ref{sec:framework} so that 
the particles in the branching diffusion process carry different marks to account for it. Throughout this section, we suppose that Assumption~\ref{ass:pde}--\ref{ass:extinction} remain valid.

\subsection{Probabilistic Representation}
\label{section:main results}

Our next assumption is the key automatic differentiation condition on the underlying diffusion process $\bX^{x}$. 
We will provide explicit conditions and formulas for it in the next sections.

\begin{assumption}
\label{ass:malliavin}
\emph{(i)} The map $x\mapsto \E[e^{-\beta {\eta^x}}h(\bX^x_{\eta^x})]$ belongs to $\cC^{1}(\cO)\cap\cC(\bar{\cO})$  and there exists a measurable function $\cW_{\del\cO}(x,W)=\cW_{\del\cO}(x,(W_r)_{r\in[0,{\eta^x}]})$ such that
\eqstar{ 
 D \E\left[ e^{-\beta {\eta^x}} h(\bX^x_{\eta^x})\right] & = \E\left[ e^{-\beta {\eta^x}} h(\bX^x_{\eta^x}) \cW_{\del\cO}(x,W)\right].
}

\noindent\emph{(ii)} For any $g:\cO\to\R$ bounded measurable, the map  $x\mapsto\E[\int_0^{{\eta^x}} {e^{-\beta s} g(\bX^{x}_{s}) \,ds}]$ belongs to $\cC^{1}(\cO)\cap\cC(\bar{\cO})$ and there exists a measurable function $\cW_{\cO}(s,x,W)=\cW_{\cO}(s,x,(W_r)_{r\in[0,s]})$ such that
\eqstar{ 
 D \E\left[ \int_0^{{\eta^x}} {e^{-\beta s} g(\bX^{x}_{s}) \,ds} \right] &=  \E\left[ \int_0^{{\eta^x}} {e^{-\beta s}g(\bX^{x}_{s}) \cW_{\cO}(s,x,W) \,ds} \right].
}
\end{assumption}

Let us define $\cW(s,x,W)=\cW(s,x,(W_r)_{r\in[0,s]})$ as follows:
\eqstar{
\cW(s,x,W) := \cW_{\del\cO}(x,W)\mathbf{1}_{\bX^{x}_s\notin\cO} + \cW_{\cO}(s,x,W)\mathbf{1}_{\bX^{x}_s\in\cO}.
}
We consider a marked branching diffusion process starting from $x\in\cO$ as in Section~\ref{section:branching diffusion recap} and denote 
\eqstar{ 
\cW_k := \mathbf{1}_{m_k=0} + \mathbf{1}_{m_k\neq0} b_{m_k}(X^k_{T_{k^-}}) \cdot \cW(\Delta T_k,X^k_{T_{k^-}},W^k),
}
where $m_k$ and $\Delta T_k := T_k - T_{k^-}$ stand for the mark and the lifetime of particle $k$ respectively. We next introduce the following random variable:
\eqstar{ 
\psi^x := \prod_{\substack{{k \in \cK^x}\\{X^k_{T_k}\notin \cO}}} \frac{e^{-\beta\Delta T_k} h(X^k_{T_k})}{\bar{F}(\Delta T_k)} \cW_k \prod_{\substack{{k \in \cK^x}\\{X^k_{T_k}\in \cO}}} \frac{\beta e^{-\beta\Delta T_k} c_{I^k}(X^k_{T_k})}{p_{I^k} \rho(\Delta T_k)} \cW_k.
}
where $\bar{F}(t):=\int_t^{\infty}{\rho(s)\,ds}$, $t\geq 0$.

\begin{proposition}
\label{prop:strongcase}
Suppose Assumption~\ref{ass:malliavin} holds. Assume further that PDE~\eqref{eq:main elliptic pde} has a solution
$u\in\cC^2(\cO)\cap\cC(\bar{\cO})$ 
such that the functions $(b_i\cdot Du)_{i=1,\cdots,m}$ are bounded and the sequence $(\psi^x_n)_{n\in\N}$ defined by  
\eqstar{
\psi^x_n := \prod_{\substack{{k \in \cup^n_{i=1}\cK^x_i}\\{X^k_{T_k}\notin \cO}}} \frac{e^{-\beta\Delta T_k} h(X^k_{T_k})}{\bar{F}(\Delta T_k)} \cW_k \prod_{\substack{{k \in \cup^n_{i=1}\cK^x_i}\\{X^k_{T_k}\in \cO}}} \frac{\beta e^{-\beta \Delta T_k} c_{I^k}(X^k_{T_k})}{p_{I^k} \rho(\Delta T_k)} \cW_k \prod_{\substack{{k \in \cK^x_{n+1}}\\{m_k = 0}}}u(X^k_{T_{k^-}})
\prod_{\substack{{k \in \cK^x_{n+1}}\\{m_k \neq 0}}} (b_{m_k} \cdot Du)(X^k_{T_{k^-}}),
}
is uniformly integrable. Then, we have $u(x) = \E[\psi^x]$.
\end{proposition}

\begin{proof}
The proof follows similar arguments as the proof of Proposition~\ref{prop:semi-linearcase}. Using It\^o's formula, we have the following Feynman-Kac representation:
\eqlnostar{eq:feynman-kac1}{
u(x) = \E\Bigl[e^{-\beta {\eta^x}}h(\bX^x_{\eta^x}) + \int^{\eta^x}_0 \beta e^{-\beta s} f(\cdot,u,Du)(\bX^x_s) \,ds\Bigr].
} 
Then, we can write
\eqlnostar{eq:feynmankacquasi}{
u(x) &= \E\left[\frac{e^{-\beta {\eta^x}}h(\bX^x_{\eta^x})}{\bar{F}({\eta^x})} \ind_{\tau \geq {\eta^x}} + \frac{\beta e^{-\beta \tau} f(\cdot,u,Du)(\bX^x_{\tau})}{\rho(\tau)}\ind_{\tau < {\eta^x}} \right]\\
&= \E\left[\frac{e^{-\beta \eta^x}h(X^{x}_{\eta^x})}{\bar{F}(\eta^x)} \ind_{\tau\geq\eta^x} + \frac{\beta e^{-\beta \tau} c_{I}(X^{x}_{\tau})}{p_{I} \rho(\tau)}u^{I_0}(X^{x}_{\tau}) \prod^m_{i=1} (b_i \cdot Du)^{I_i}(X^{x}_{\tau})\ind_{\tau<\eta^x} \right].
}
In other words, we have $u(x)=\E\left[\psi_0^x\right]$. 
Furthermore, in the original Feynman-Kac formula, we obtain by differentiating and using Assumption \ref{ass:malliavin},
\eqlnostar{eq:feynman-kac deriv1}{
D u(x) &= \E\left[e^{-\beta {\eta^x}}h(\bX^x_{\eta^x}) \cW_{\del\cO}(x,W) + \int^{\eta^x}_0 {\beta e^{-\beta s}f(\cdot,u,Du)(\bX^x_s) \cW_{\cO}(s,x,W) \,ds}\right]\\
&= \E\bigg[\psi_0^x \MalWt(T_{\emptyset},x,W) \bigg].
}
Since each offspring has the same dynamic as the parent particle, we can repeat the above calculations for $k\in\cK^x_1$ and plug the results back in \eqref{eq:feynmankacquasi} to obtain $u(x)=\E[\psi_1^x]$ by conditional independence of particles in $\cK^x_1$ given $\cF_0$.
We conclude by iteration that, for any $n\in\N$, $u(x)=\E[\psi_n^x]$ and, as $n \to \infty,$ $u(x)=\E[\psi^x]$.
\end{proof}

Similar to Section~\ref{sec:mainresultsemi}, Proposition~\ref{prop:strongcase} provides a result of uniqueness for a class of semi-linear PDEs given an appropriate uniform integrability condition is satisfied. Next we establish a result of existence by showing that the probabilistic representation is a viscosity solution of PDE~\eqref{eq:main elliptic pde}. This is the main result of this section.

\begin{theorem}
\label{thm:quasilinear}
Suppose Assumption~\ref{ass:malliavin} holds. If we further assume that $(\psi^x)_{x\in\cO}$ and, for $i=1,\ldots,m$, $(\psi^x b_i(x)\cdot\MalWt(T_{\emptyset},x,W))_{x\in\cO}$ are uniformly bounded in $L^1$, then $u:x\mapsto \E[\psi^x]$ belongs to $\cC^1(\cO)\cap\cC(\bar{\cO})$ and solves PDE~\eqref{eq:main elliptic pde} in the viscosity sense.
\end{theorem}

\begin{proof}
The proof follows similar arguments as the proof of Theorem~\ref{thm:semi-linear}. We observe first that, by definition of $u$, it holds
\eqstar{
u(x) 
& = \E\left[\frac{e^{-\beta \eta^x} h(X^{x}_{\eta^x})}{\bar{F}(\eta^x)} \ind_{\tau\geq\eta^x} + \frac{\beta e^{-\beta \tau} c_{I}(X^{x}_{\tau})}{p_{I} \rho(\tau)}\prod_{i=0}^{\lvert I\rvert-1} \psi^{X^{x}_{\tau}}_i \ind_{\tau<\eta^x} \right].
}
where
\eqstar{
  \psi^{X^{x}_{\tau}}_i := \prod_{\substack{{k=(i,\ldots)\in\cK^x}\\{X^{k}_{T_{k}}\notin \cO}}} \frac{e^{-\beta\Delta T_{k}} h(X^{k}_{T_{k}})}{\bar{F}(\Delta T_{k})} \cW_{k} \prod_{\substack{{k=(i,\ldots)\in\cK^x}\\{X^{k}_{T_{k}}\in \cO}}} \frac{\beta e^{-\beta\Delta T_{k}} c_{I^{k}}(X^{k}_{T_{k}})}{p_{I^{k}} \rho(\Delta T_{k})} \cW_{k}.
}
It further follows from the branching property that, conditioned on $\cF_0$, $(\psi^{X^{x}_{\tau}}_i)_{i=0,\ldots,\lvert I \rvert-1}$ are independent random variables, among which the first $I_0$ are identical in law to $\psi^{X_0}$, the next $I_1$ are identical in law to $\psi^{X_0} b_1(x)\cdot\MalWt(T_{\emptyset},x,W)$ and so on, where $X_0$ is distributed as $X^{x}_{\tau}$ and independent of $\cF_n$ for all $n\in\N$. Through this argument, we deduce that
\eqstar{
  \E\Big[\prod_{i=0}^{\lvert I\rvert-1} \psi^{X^{x}_{\tau}}_i \,\Big|\, \cF_0  \Big] \ind_{\tau < \eta^x} = u^{I_0}\left(X^{x}_{\tau}\right) \prod^m_{i=1} v_i^{I_i}(X^{x}_{\tau}) \ind_{\tau < \eta^x}.
}
where $(v_i)_{i=1,\ldots,m}$, $v_i:\cO\mapsto\R$ are defined as
\eqstar{
v_i(x) := \E[\psi^x b_i(x)\cdot\MalWt(T_{\emptyset},x,W)].
}
Working backward along the lines of the proof of Proposition~\ref{prop:strongcase}, we deduce that 
\eqlnostar{eq:feynmankac}{
u(x) = \E\Bigl[e^{-\beta \eta^x}h(\bX^x_{\eta^x}) + \int^{\eta^x}_0 \beta e^{-\beta s} \Bigl(\sum_{l \in L} c_l u^{l_0} \prod^m_{i=1} v_i^{l_i}\Bigr)(\bX^x_s) \,ds\Bigr].
}
In particular, since $u$ and $(v_i)_{i=1,\ldots,m}$ are bounded by assumption, it follows from Assumption~\ref{ass:malliavin} that $u$ belongs to $\cC^1(\cO)\cap\cC(\bar{\cO})$ and 
\eqstar{
D u(x) & = \E\Bigl[e^{-\beta {\eta^x}}h(\bX^x_{\eta^x}) \cW_{\del\cO}(x,W) + \int^{\eta^x}_0 \beta e^{-\beta s} \Bigl(\sum_{l \in L} c_l u^{l_0} \prod^m_{i=1} v_i^{l_i}\Bigr)(\bX^x_s) \cW_{\cO}(s,x,W)\, ds\Bigr] \\
 & = \E[\psi^x \MalWt(T_{\emptyset},x,W)].
}
Thus, for all $i=1,\ldots,m$, $v_i$ coincides with $b_i\cdot Du$ and~\eqref{eq:feynmankac} reads as
\eqstar{
u(x) = \E\Bigl[e^{-\beta {\eta^x}}h(\bX^x_{\eta^x}) + \int^{\eta^x}_0 \beta e^{-\beta s} f(\cdot,u,Du)(\bX^x_s)\, ds\Bigr].
}
The fact that $u$ is a viscosity solution of PDE~\eqref{eq:main elliptic pde} now follows by classical arguments.
\end{proof}

\subsection{Automatic Differentiation Formula: the General Case}
\label{sec:nDmalliavin weight}

The aim of this section is to provide sufficient conditions to ensure that Assumption~\ref{ass:malliavin} holds and to derive explicit formula for $\cW$. 
The automatic differentiation formula discussed in the following originates from Thalmaier~\cite{thalmaier1997adf} and was subsequently developed by Delarue~\citep{delarue2003estimates} and Gobet~\citep{gobet2004revisiting}.

\begin{assumption}
  \label{ass:autodiff}
   \emph{(i)} The coefficients $(\mu,\sigma)$ belong to $\cC^{1,\alpha}(\bar{\cO})$.\\
   \emph{(ii)} The diffusion coefficient $\sigma$ is uniformly elliptic.\\
   \emph{(iii)} The boundary $\del \cO$ is of class $\cC^2$.\\
 \emph{(iv)}  The function $h$ can be extended to a function of class $\cC^{1,\alpha}$ on $\bar{\cO}$.
\end{assumption}

We start by establishing a technical lemma. Fix a finite horizon $T>0$ and denote for any $s>0$, 
\eqstar{
\theta_{s}(r,y):=\frac{1}{d\left(y,\partial\cO\right)^2\left(s-r\right)},\quad \text{for all }y\in\cO,\, r\in [0,s).
}

\begin{lemma}
\label{lem:delarue}
Under Assumption~\ref{ass:autodiff}, it holds for all $x\in\cO$ and $s > 0$,
\eqlnostar{eq:delarue1}{
\int_0^{{\eta^x}\wedge s} {\theta_s(r,\bX_r^x)\, dr} = \infty,\quad\P-\text{a.s.}
}
In addition, if we denote
\eqstar{
\zeta_s := \inf {\left\{ t>0\,:\  \int_0^t {\theta_s(r,\bX_r^x) \,dr} = 1\right\}},
}
then there exists $t<s$ such that $\zeta_s\leq {\eta^x}\wedge t$ and for all $q\geq  1$,

\eqlnostar{eq:delarue2}{
\E\left[\left(\int_0^{\zeta_{s\wedge T}} {\theta^2_{s\wedge T}(r,\bX_r^x) \,dr}\right)^q\right]\leq \frac{C}{d\left(x,\partial\cO\right)^{4q-2} (s\wedge T)^q},
}
where $C>0$ depends on $q$ and $T$ but not on $x$ or $s$.
\end{lemma}

\begin{proof}
The proof essentially follows from Delarue~\cite{delarue2003estimates}. 
Indeed, if $\cO$ is a ball, both identities~\eqref{eq:delarue1} and~\eqref{eq:delarue2} are easily obtained by repeating the arguments of Propositions 2.3 and 2.4 in~\cite{delarue2003estimates} while working with $(r,y)\mapsto d(y,\partial\cO)\sqrt{s-r}$ instead of $(r,y)\mapsto d(y,\partial\cO)(s-r)$. 
For an arbitrary domain, it suffices to work with a $\cC^2$--extension of the distance to the boundary (see, \textit{e.g.}, Gilbarg and Trudinger~\cite[Lemma 14.16]{gilbarg2015elliptic}). Additionally, it follows from~\eqref{eq:delarue1} that $\zeta_s\leq {\eta^x}\wedge s$. Furthermore, it holds for all $t < s$,
\eqstar{
\int_0^{t} {\theta_s(r,\bX_r^x)\, dr} \,\ind_{t\leq{\eta^x}} \geq - C^{-1} \log\left(1-\frac{t}{s}\right) \ind_{t\leq{\eta^x}},
}
where $C:=\mathrm{diam}(\cO)^2/4$. Thus for $t=(1-e^{-C})s$, we have $\zeta_s\leq {\eta^x}\wedge t$. 
\end{proof}

\begin{proposition}
\label{prop:autodiff}
Under Assumption~\ref{ass:autodiff}, the assertions of Assumption~\ref{ass:malliavin} are satisfied with
\eqstar{
\cW^{\top}_{\partial\cO}\left(x,W\right) & = \int_0^{\zeta_T}{\theta_T(r,\bX_r^x) \left( \sigma^{-1}(\bX_r^x) Y_r^x \right)^{\top}\, dW_r}, \\
\cW^{\top}_{\cO}\left(s,x,W\right) & = \int_0^{\zeta_{s\wedge T}}{\theta_{s\wedge T}(r,\bX_r^x)\left( \sigma^{-1}(\bX_r^x) Y_r^x \right)^{\top}\, dW_r},
}
where $Y^x$ is the tangent process given by 
\eqstar{
Y_s^x = I_d + \int_0^s {D\mu(\bX_r^x) Y_r^x \,dr} + {\sum_{i=1}^d \int_0^s D\sigma_i(\bX_r^x) Y_r^x \,dW^i_r}.
}
and $\sigma_i$ denotes the $i$th column of $\sigma$. In addition, it holds for all $q\geq 1$,
\eqlnostar{eq:estimates}{
\E\left[\left|\int_0^{\zeta_{s\wedge T}} {\theta_{s\wedge T}(r,\bX_r^x) \left( \sigma^{-1}(\bX_r^x) Y_r^x \right)^{\top}\, dW_r}\right|^{q}\right]\leq \frac{C}{\dd(x, \del \cO)^{2q-1} (s\wedge T)^{\frac{q}{2}}}.
}
\end{proposition}

This result is a slight extension of the automatic differentiation formula obtained by Delarue~\cite{delarue2003estimates} and Gobet~\cite{gobet2004revisiting}. We postpone the proof to Appendix~\ref{proof_autodiff}. Note that estimate~\eqref{eq:estimates} is an easy consequence of Lemma~\ref{lem:delarue}. Indeed, successively using the Burkh\"older-Davis-Gundy inequality, the boundedness of $\sigma^{-1}$ and the fact that $\sup_{0 \leq r \leq T} \{|Y^x_r|\}$ has finite moments, we derive for all $s\leq T$,
\eqstar{
\E \Bigl[ \big \vert \int_0^{\zeta_s} {\theta_s(r,\bX_r^x)\left( \sigma^{-1}(\bX_r^x) Y_r^x \right)^{\top}\, dW_r} \big \vert^{q} \Bigr]
&\leq C \E \Bigl[\sup_{0 \leq r \leq T} \{\left|Y_r\right|^{q}\} \Big(\int^{\zeta_s}_0 \theta_s^2(r,\bX^x_r) d r\Big)^{\frac{q}{2}}\Bigr]\\
&\leq C\E \Bigl[\bigl(\int^{\zeta_s}_0 \theta_s^2(r,\bX^x_r) d r \bigr)^q\Bigr]^{1/2}\\
&\leq \frac{C}{d(x, \del \cO)^{2q-1} s^{\frac{q}{2}}}.
}

\begin{remark}
The choice of weight function in the automatic differentiation formula is not unique. Actually we can replace $(\theta_s(r,\bX^x_r))_{r\geq 0}$ by any predictable process $(\tilde{\theta}_s(r))_{r\geq 0}$ satisfying appropriate integrability conditions and such that 
\eqstar{
\tilde{\theta}_s(r)=0~~\text{for all }r\geq {\eta^x}\wedge s \quad\text{ and }\quad \int_0^{{\eta^x}\wedge s} {\tilde{\theta}_s(r)\, dr} = 1.
}
We refer to Delarue~\cite{delarue2003estimates} for more examples.
\end{remark}

We now provide sufficient conditions to ensure that $(\psi^x)_{x\in\cO}$ and $(\psi^x b_i(x)\cdot\MalWt(T_{\emptyset},x, W))_{x\in\cO}$, $i=1,\ldots,m$,  are uniformly bounded in $L^q$ for $q\geq 1$. Let us define
\eqstar{
C_{1,q} & := \sup_{t\geq 0} {\left\{\frac{e^{-\beta t}}{\bar{F}(t)}\right\}} \sup_{x\in\del\cO} {\left\{|h(x)|\right\}} \sup_{x\in\cO, i=1,\ldots,m} {\left\{\E\left[\left|b_i(x)\cdot\cW_{\del\cO}\left(x,W\right)\right|^q\right]^{\frac{1}{q}} \vee 1\right\}}, \\
C_{2,q} & := \sup_{x\in\cO ,\, l\in L} {\left\{\frac{\left|c_l(x)\right|}{p_l} \right\}} \sup_{t\geq 0,\, x\in\cO ,\, i=1,\ldots, m} {\left\{\frac{\beta e^{-\beta t}}{\rho(t)} \left(\E\left[\left|b_i(x)\cdot\cW_{\cO}\left(t,x,W\right)\right|^q\right]^{\frac{1}{q}} \vee 1\right)\right\}}.
}
In view of~\eqref{eq:estimates}, $C_{1,q}$ and $C_{2,q}$ are clearly finite if the following condition is satisfied:
\eqstar{
\sup_{t\geq 0} {\left\{\frac{e^{-\beta t}}{\bar{F}(t)}\right\}} \vee \sup_{t\geq 0} {\left\{\frac{\beta e^{-\beta t}}{\sqrt{t} \rho(t)}\right\}} \vee \sup_{l\in L} \left\{\frac{\|c_l\|_{\infty}}{p_l}\right\} \vee \sup_{x\in\cO,\, i=1,\ldots, m} {\left\{\frac{\left|b_i(x)\right|}{d(x,\del\cO)^{2-\frac{1}{q}}}\right\}}  < \infty.
}
Unlike in the case of linear gradient term,  we need to chose a lifetime distribution $\rho(t)\neq\beta e^{-\beta t}$ to ensure that the first two terms in the above expression are finite.
For instance, we can take a Gamma distribution with shape parameter $0.5$ and rate parameter $\beta'<\beta$, \textit{i.e.}, 
$\rho(t)=\sqrt{\beta'/(\pi t)} e^{-\beta' t}$.
Another choice consists in taking a generalized gamma distribution of the form $\rho(t) = \beta'/(2\sqrt{t}) e^{-\beta' \sqrt{t}}$ for any $\beta'>0$.

Let us denote 
\eqstar{
\delta := 1-\inf_{x\in\cO}\left\{\E\left[\bar{F}\left(\eta^x\right)\right]\right\}.
}
Similar to Section~\ref{sec:integrability}, we then introduce a new probability mass function $(\tilde{p}_l)_{l\in \tilde{L}}$ with $\tilde{L}:=L\cup\{0_{m+1}\}$ where $0_{m+1}$ is the zero element of $\N^{m+1}$ as follows:
\eqstar{
\tilde{p}_{0_{m+1}} := 1 - \delta + \delta p_{0_{m+1}} ~~ \text{and} ~~ \tilde{p}_l &:= \delta p_l~~\text{for all }l\neq 0_{m+1},
}
and the corresponding transition matrix $\tilde{P}=(\tilde{P}_{i,j})_{i,j\geq 0}$ given by
\eqstar{
\tilde{P}_{0,0} = 1 ~~ \text{and} ~~ \tilde{P}_{i,i+j-1} = \sum_{|l|=j} \tilde{p}_l ~~ \text{for all } i\geq 1,\, j\geq 0.
}

\begin{proposition}
\label{prop:integrability_quasi}
Denote $C_{0,q}:=C_{1,q}\vee C_{2,q}$.
\newline
\noindent (i) If $C_{0,q}\leq 1$, then $(\psi^x)_{x\in\cO}$ and $(\psi^x b_i(x)\cdot\MalWt(T_{\emptyset},x, W))_{x\in\cO}$, $i=1,\ldots,m$, are bounded by $1$ in $L^q$. 
\newline
\noindent (ii) Denote by $R$ the common radius of convergence of $f(s):=\sum_{l\in L} {p_l s^{|l|}}$ and $\tilde{f}(s):=\sum_{l\in \tilde{L}} {\tilde{p}_l s^{|l|}}$. 
If $R>1$ and $\sum_{l\in\tilde{L}}{|l|\tilde{p}_l}<1$, then 
$(\psi^x)_{x\in\cO}$ and $(\psi^x b_i(x)\cdot\MalWt(T_{\emptyset},x,W))_{x\in\cO}$, $i=1,\ldots m$, are uniformly bounded in $L^q$ provided that $ C^q_{0,q} \leq \gamma$ where
\eqstar{
\gamma := \frac{s^*}{\tilde{f}(s^*)}=\frac{s^*}{1-\delta + \delta f(s^*)},
}
where $s^*$ is the solution (if any) to $s\tilde{f}'(s)=\tilde{f}(s)$ and $s^*=R$ otherwise. In addition, if $R=+\infty$, then $\gamma$ goes to infinity as $\mathrm{diam}(\cO)$ goes to zero.
\end{proposition}

\begin{proof}
Let us denote by $\mathcal{G}$ the following $\sigma$-algebra:
\eqstar{
\mathcal{G}:=\sigma\left(\tau^k, I^k, k\in\K\right).
}
Conditioning by $\mathcal{G}$, we obtain for all $x\in\cO$,
\eqstar{
\E\left[\left|\psi^x\right|^q\right] \leq \E\left[C_{0,q}^{q\left|\cK^x\right|}\right],
}
and the same inequality holds if we consider $\psi^x b_i(x)\cdot\MalWt(T_{\emptyset},x, W)$ instead of $\psi^x$. Part~(i) of Proposition~\ref{prop:integrability_quasi} follows immediately. As for Part~(ii), the desired result follows by the same arguments as in Proposition~\ref{lem:unifinteasy}.
\end{proof}

\subsection{Automatic Differentiation Formula: the One-Dimensional Case}
\label{sec:1Dmalliavin weight}

The automatic differentiation formula given by Proposition~\ref{prop:autodiff}, though quite general, has several drawbacks for numerical applications. Indeed one needs to compute a stochastic integral for each particle that holds a non-zero mark. In addition, the weight function $\cW(s,x,W)$ explodes as $s$ goes to $0$ or $x$ approaches the boundary. In this section, we provide a simpler automatic differentiation formula that is satisfied for a one-dimensional Brownian motion exiting from an interval.

Let us assume throughout this section that $d=1$, $b=0$, $\sigma=1$ and $\cO=(-r,r)$ for some $r>0$. We denote 
\eqstar{
W^x_s := x + W_s,\quad s\geq 0,\,x\in\cO.
}
We also introduce the function $\mathfrak{W}:\cO\x \bar{\cO} \to \R$ as
\eqstar{
\fW(x,y) :=
\begin{cases}
\frac{\sqrt{2\beta}}{\tanh\left(\sqrt{2\beta}(x+r)\right)}, & \text{if } y>x, \\
\frac{\sqrt{2\beta}}{\tanh\left(\sqrt{2\beta}(x-r)\right)}, & \text{if } y\leq x.
\end{cases}
}
Both lemmas below show that Assumption~\ref{ass:malliavin} is satisfied in this setting.

\begin{lemma}
\label{lem:adfin1d}
For any $g:\cO\to\R$ bounded measurable, the map  $x\mapsto\E[\int_0^{{\eta^x}} {e^{-\beta s} g(W^{x}_{s}) \,ds}]$  belongs to $\cC^1(\cO)\cap\cC(\bar{\cO})$ and 
\eqstar{
D \E\left[ \int_0^{{\eta^x}} {e^{-\beta s} g(W^{x}_{s}) \,ds} \right] &=  \E\left[ \int_0^{{\eta^x}} {e^{-\beta s}g(W^{x}_{s}) \fW(x,W^x_s) \,ds} \right].
}
\end{lemma}

\begin{proof}
Let us denote 
\eqstar{
\chi(x) := \E\left[ \int_0^{{\eta^x}} {e^{-\beta s} g(W^{x}_{s}) \,ds} \right].
}
We first assume that $g$ is continuous. Then $u$ satisfies the following ODE:
\eqlnostar{eq:ODE0}{
\frac{1}{2} \chi'' - \beta \chi + g = 0 ~~ \text{in } \cO, \quad~ \chi = 0~~ \text{on } \del\cO.
}
By standard arguments, including variation of parameters, 
we deduce that
\eqlnostar{eq:green1d}{
\chi(x) = \int_{-r}^{r} {G(x,y) g(y) \,dy},
}
where $G:\bar{\cO}\times \cO\to \R$ is given by
\eqstar{
G(x,y) := \frac{-2}{\sqrt{2\beta}} \left(\sinh\left(\sqrt{2\beta}(x-y)^+\right) - \frac{\sinh\left(\sqrt{2\beta}(r+x)\right)}{\sinh\left(2\sqrt{2\beta}r\right)} \sinh\left(\sqrt{2\beta}(r-y)\right)\right).
}
By a monotone class argument, \eqref{eq:green1d} remains valid if $g$ is only assumed bounded measurable. Then the continuity of $\chi$ follows immediately from the dominated convergence theorem.
Further, a direct calculation yields that 
\eqstar{
\frac{\partial_x G(x,y)}{G(x,y)} =
\begin{cases}
\frac{\sqrt{2\beta}}{\tanh\left(\sqrt{2\beta}(x+r)\right)}, & \text{if } y>x, \\
\frac{\sqrt{2\beta}}{\tanh\left(\sqrt{2\beta}(x-r)\right)}, & \text{if } y<x.
\end{cases}
}
The desired result then follows by differentiation under the integral sign in~\eqref{eq:green1d}.
\end{proof}

\begin{lemma}
\label{lem:adfout1d}
The map  $x\mapsto\E[e^{-\beta {\eta^x}} h(W^{x}_{{\eta^x}})]$  belongs to $\cC^1(\cO)\cap\cC(\bar{\cO})$ and 
\eqstar{ 
D \E\left[ e^{-\beta {\eta^x}} h(W^{x}_{{\eta^x}}) \right] &=  \E\left[ e^{-\beta {\eta^x}} h(W^{x}_{{\eta^x}}) \fW(x,W^x_{\eta^x}) \right].
}
\end{lemma}

\begin{proof}
Let us denote 
\eqstar{
\varphi(x) := \E\left[ e^{-\beta {\eta^x}} h(W^{x}_{{\eta^x}}) \right].
}
It satisfies the following ODE:
\eqlnostar{eq:ODE}{
\frac{1}{2} \varphi'' - \beta \varphi = 0 ~~ \text{in } \cO, \quad~ \varphi=h ~~ \text{on } \del\cO.
}
It follows that
\eqstar{
\varphi(x) = H(x,r) h(r) + H(x,-r) h(-r),
}
where $H:\cO\x\del\cO\to\R$ is given by 
\eqstar{
H(x,y) := \frac{\sinh(\sqrt{2\beta}(x+y))}{\sinh(2\sqrt{2\beta}y)}.
}
In particular, $\varphi$ clearly belongs to $\cC^1(\cO)\cap\cC(\bar{\cO})$. 
To conclude, it remains to observe that
\eqstar{
\frac{\partial_x H(x,y)}{H(x,y)} = \frac{\sqrt{2\beta}}{\tanh(\sqrt{2\beta}(x+y))}.
}
\end{proof}

\begin{remark}
We observe that the weight function still explodes when $x$ approaches the boundary as in Section~\ref{sec:nDmalliavin weight}.
In view of the proof of Lemma~\ref{lem:adfin1d}, this feature cannot be avoided since $\chi'$ satisfies an ODE of type~\eqref{eq:ODE0} with non-zero boundary conditions.
As a consequence, we need to assume once again that the non-linear gradient term vanishes at the boundary of the domain in order to control the explosion of the weight function. 
\end{remark}

Since the weight function does not explode in time anymore, we can work with $\rho(t)=\beta e^{-\beta t}$ as in Section~\ref{sec:semilinear}. Thus we have 
\eqstar{ 
\psi^x := \prod_{\substack{{k \in \cK^x}\\{X^k_{T_k}\notin \cO}}} h(X^k_{T_k}) \cW_k \prod_{\substack{{k \in \cK^x}\\{X^k_{T_k}\in \cO}}} \frac{c_{I^k}(X^k_{T_k})}{p_{I^k}} \cW_k,
}
where $\cW_k$ are constructed from the formulas of Lemmas~\ref{lem:adfin1d}--\ref{lem:adfout1d}. 
We now provide sufficient conditions to ensure that $(\psi^x)_{x\in\cO}$ and  $(\psi^x b_i(x)\cdot\MalWt(T_{\emptyset},x, W))_{x\in\cO}$, $i=1,\ldots,m$,  are uniformly bounded in $L^q$. Let us define
\eqstar{
C_1 & := \sup_{x\in\cO,\, y\in\bar{\cO},\, i=1,\ldots, m} {\left\{\left|b_i(x) \fW\left(x,y\right)\right|\right\}} \vee 1.
}
Clearly, the constant $C_1$ is finite if the following condition is satisfied:
\eqlnostar{eq:1d quasilinear integrable}{
\sup_{x\in\cO,\, i=1,\ldots, m} {\left\{\frac{\left|b_i(x)\right|}{d(x,\del\cO)}\right\}}  < \infty.
}
Now we consider the family $(\bar{\psi}^x)_{x\in\cO}$ given by
\eqstar{
\bar{\psi}^x := \prod_{\substack{{k \in \cK^x}\\{X^k_{T_k}\notin \cO}}} C_1 h(X^k_{T_k}) \prod_{\substack{{k \in \cK^x}\\{X^k_{T_k}\in \cO}}} C_1 \frac{c_{I^k}(X^k_{T_k})}{p_{I^k}}.
}
The tools developed in Section~\ref{sec:integrability} allow to derive sufficient conditions to ensure that the family $(\bar{\psi}^x)_{x\in\cO}$ is bounded in $L^q$ . Since $|\psi^x|\leq |\bar{\psi}^x|$, the boundedness on $(\psi^x)_{x\in\cO}$ in $L^q$ follows immediately. The same holds for $(\psi^x b_i(x)\cdot\MalWt(T_{\emptyset},x,W))_{x\in\cO}$ for $i=1,\ldots,m$.

\section{Numerical Examples with Related Technical Discussions}
\label{section: examples}
In this section we demonstrate the application of our theoretical results with the help of several examples. We restrict our attention to semi-linear elliptic PDEs over a rectangular domain where the underlying diffusion is a Brownian motion. This allows us to generate unbiased samples for exit time and position using the walk on squares scheme as implemented in the numerical library \texttt{exitbm} developed by Lejay~\cite{lejay:inria-00561409}. We note that our theoretical results are applicable more widely but the numerical estimates for semi-linear PDEs driven by Brownian motion over rectangular domain do not involve any bias. We performed all the numerical computations on a machine with 2,5 GHz Intel Core i5 processor and 4GB RAM.

\subsection{Example 1}
\label{sec:cosh} 
We begin with the following semi-linear PDE in one-dimension (ODE)
\eqlnostar{eq:example1 semi-linear}{
u'' - u + u^3 = 0,
}
with explicit solution $u(x) = \frac{\sqrt{2}}{\cosh(x)}.$ To test the validity of our probabilistic representation in Theorem~\ref{thm:semi-linear}, we reformulate the above PDE, in the form \eqref{eq:semi-linearpde}, as following
\eqlnostar{eq:ode example 1}{
 \frac{1}{2}u'' + \left(\frac{1}{2}u^3 + \frac{1}{2}u - u\right) =0 ~~ \text{in } (-r,r), \quad~ u(x) = \frac{\sqrt{2}}{\cosh(x)},\, x \in \{-r,r\},
}
For domain $\cO = (-r,r),$ we know from standard literature that the first eigenvalue of Laplacian is given as $\lambda_1 = \tfrac{\pi^2}{4 r^2}.$ Then, by choosing $\beta = 1,$ $p_1 = \tfrac{1}{2}$ and $p_3 = \tfrac{1}{2},$ the criteria for almost sure extinction of the branching Brownian motion \eqref{eq:extinction} holds for any $r \leq \tfrac{\pi}{\sqrt{8}}$. In addition, the integrability condition as presented in Proposition \ref{prop:unifint} is satisfied for $q=1$ by $v(x) = \frac{\sqrt{2}}{\cosh(x)}.$  Therefore, our probabilistic representation of the solution given as 
\eqlnostar{eq:prob rep example1}{
u(x)= \E\Biggl[\prod_{\substack{{k \in \cK^x}\\{X^k_{T_k}\notin \cO}}} h(X^k_{T_k})\Biggr],
}
is valid for any $r \leq  \tfrac{\pi}{\sqrt{8}}.$ Once more, we make use of Proposition \ref{prop:unifint} and see that for $q=2,$ it is satisfied by $v(x) = \sqrt{1+\lambda} \cos\bigl( \sqrt{\lambda} x \bigr)$ with $\lambda =6$ and $ r \leq 0.338.$ This implies that $(\psi^x)_{x \in \cO}$ is uniformly bounded in $L^2$ and we are certain that it has finite variance. In Table \ref{tab:example1}, for $r=0.3,$ we illustrate numerical results for different values of starting position $x$ with 1$\times 10^6$ Monte Carlo sample paths which exhibits the accuracy of our estimator.

\begin{table}[H]
\centering
\small
\begin{tabular}{c c c c c c}
\toprule
$x$ & Estimate & 99\% conf. interval & Std. Dev./Mean & Relative error & Run time (secs)\\
\toprule
    0 & 1.4144 & [1.4134, 1.4153] & 0.2644 & 0.0112\% & 13\\
-0.2 & 1.3859 & [1.3852, 1.3866] & 0.1872 & 0.0358\% & 26\\
\hline
\end{tabular}
\caption{{\em Numerical results for the unbiased estimator of probabilistic representation of one-dimensional semi-linear PDE \eqref{eq:ode example 1}.}}
\label{tab:example1}
\end{table}
We observe that even for larger values of $r,$ our probabilistic representation provides an accurate estimator but with a standard deviation estimate which converges slowly. In Table \ref{tab:example1r}, we vary the size of the interval to test the accuracy and notice that as more branching particles are generated for larger intervals, the estimator has higher variance and run time.
\begin{table}[H]
\centering
\small
\begin{tabular}{c c c c}
\toprule
$r$ & 99\% conf.  & Std. Dev./Mean & Run time  \\
       & interval 	   &  							  & (secs) 		\\
\toprule
0.4 & [1.413654, 1.417886] & 0.5793  & 17\\
0.5 & [1.410199, 1.417666] & 0.9134  & 23\\
\hline
\end{tabular}
\caption{{\em Numerical results for the unbiased estimator for different values of $r$ and $x=0.$}}
\label{tab:example1r}
\end{table}
\begin{remark}
\label{remark:cosh_example}
If $r\geq \arcosh(\sqrt{2})$, the solution provided by the probabilistic representation takes values in $[0,1]$ and thus it does not coincide with $u(x)=\tfrac{\sqrt{2}}{\cosh(x)}$. In particular, in view of Proposition~\ref{prop:semi-linearcase}, it turns out that the sequence $(\psi_n^x)_{n\in\N}$ is not uniformly integrable for $\arcosh(\sqrt{2})\leq r\leq \tfrac{\pi}{\sqrt{8}}$. In Table \ref{tab:example09}, we present the results which support our observation.

\begin{table}[H]
\centering
\small
\begin{tabular}{c c c c c}
\toprule
$x$ & Estimate & 99\% conf. interval & Std. Dev./Mean & Run time (secs)\\
\toprule
0 & 0.9597 & [0.9595, 0.9600] & 0.0746 & 123\\
-0.2 & 0.9612 & [0.9611, 0.9614] & 0.0722 & 121\\
\hline
\end{tabular}
\caption{{\em Numerical results for the unbiased estimator of probabilistic representation for domain $(-0.9,0.9).$}}
\label{tab:example09}
\end{table}
\end{remark}

\subsection{Example 2}
We consider another semi-linear PDE  in one dimension. It is given as 
\eqlnostar{eq:ode example 2}{
\frac{1}{2} u'' + \left( \frac{1}{2}  -\frac{3}{2}u^2 - u\right) = 0,
}
with solution $u(x) = 1 + 2 \tan^2(x).$ Then, for domain $\cO = (-r,r)$, $\beta = 1$ and probability mass function as $p_0= 0.25$ and $p_2= 0.75,$ the branching Brownian motion goes extinct almost surely for $r \leq \tfrac{\pi}{2}.$ Next, to verify that our probabilistic representation holds, we work with the conditions in Proposition \ref{prop:extinction}. We first compute 
\eqstar{
\tilde{p}_0 := p_0 + \left(1-p_0\right)\inf_{x\in\cO}\left\{\E\left[e^{-\eta^x}\right]\right\} ~~ \text{and} ~~ \tilde{p}_2 &:= p_2 \left(1-\inf_{x\in\cO}\left\{\E\left[e^{-\eta^x}\right]\right\}\right).
}
From standard calculations, we get
\eqstar{
\inf_{x\in\cO} {\left\{\E\left[e^{-\eta^x}\right]\right\}} =\inf_{x\in\cO} {\left\{\frac{\cosh (\sqrt{2 } x)}{\cosh (\sqrt{2 } r)}\right\}} = \frac{1}{\cosh (\sqrt{2} r)} .
}
This gives us $
\tilde{p}_0 =  0.25 + \tfrac{0.75}{\cosh (\sqrt{2} r)}$ and $\tilde{p}_2 = 0.75 \left(1 - \tfrac{1}{\cosh (\sqrt{2} r)}\right).$ In addition, the integrability constant $C_0 =  \max\{2,1+ 2 \tan^2(r)\}.$ Next, in view of~\eqref{eq:limit threshold}, the choice of threshold is given as $\gamma = \tfrac{1}{\sqrt{4\tilde{p}_2(1-\tilde{p}_2)}}.$ Then, for our probabilistic representation of the solution to be valid, we need that $C_0 < \gamma$, which gives us that $r \leq 0.31.$ As in the previous example, we would like to compute a confidence interval for the estimate of the solution. Once again, by checking the conditions in Proposition \ref{prop:extinction}, we get that for our estimator $(\psi_x)_{x \in \cO}$ to be uniformly bounded in $L^2,$ we must choose $r \leq 0.146.$ We set $r=0.14,$ and obtain the following results for $1 \times 10^6$ Monte Carlo sample paths in Table \ref{tab:example2} which exhibits the accuracy of our estimator.

\begin{table}[H]
\centering
\small
\begin{tabular}{c c c c c c}
\toprule
$x$ & Estimate & 99\% conf. interval & Std. Dev./Mean & Relative error & Run time (secs)\\
\toprule
0     & 0.9999 & [0.9989, 1.0001] & 0.3931 & 0.0097\% & 10\\
-0.1 &  1.0198 & [1.0190, 1.0205] & 0.2745 & 0.0351\% & 12\\
\hline
\end{tabular}
\caption{{\em Numerical results for the unbiased estimator of probabilistic representation of one-dimensional semi-linear PDE \eqref{eq:ode example 2}.}}
\label{tab:example2}
\end{table}

In the course of our experiments, we chose different values of $r$ and observed that  even for higher values, we are able to obtain accurate estimates of the solution. For $r=0.3,$ we present the results in Table \ref{tab:example2_1}.
\begin{table}[H]
\centering
\small
\begin{tabular}{c c c c c c}
\toprule
$x$ & Estimate & 99\% conf. interval & Std. Dev./Mean & Relative error & Run time (secs)\\
\toprule
0.0  & 1.0006 & [0.9971, 1.0040] & 1.3378 & 0.0575\% & 12\\
-0.1 & 1.0203 & [1.0168, 1.0234] & 1.3641 & 0.0209\% & 17\\
\hline
\end{tabular}
\caption{{\em Numerical results for the unbiased estimator of probabilistic representation of one-dimensional semi-linear PDE \eqref{eq:ode example 2}.}}
\label{tab:example2_1}
\end{table}

\subsection{Example 3}
Next, we consider a multidimensional semi-linear PDE  
\eqlnostar{eq:example3 semi-linear}{
\Delta u = 2 d (u^3+u),
}
with an explicit solution $u(x) = \tan(\sum^d_{i=1}x_i)$ where $x = (x_1,\ldots,x_d)\in\R^d.$ We reformulate the above PDE as
\eqstar{
 \frac{1}{2}\Delta u + d \left(-u^3 - u\right) =0 ~~ \text{in } \cO,\quad~ u = \tan\Big(\sum^d_{i=1}x_i\Big) , \, x\in \partial\cO,
}
with $\cO = (-r,r)^d$ for $r > 0$ and set $\beta = d$ and $p_3=1.$ With $\lambda_1 = \tfrac{d \pi^2}{4 r^2},$ the branching Brownian motion goes extinct for $r \leq \tfrac{\pi}{4}.$ 

We first consider $d=2.$ For $r < \tfrac{\pi}{8},$ the integrability constant $C_0 < 1$ and thus our estimator will have finite moments. In order to check if our probabilistic representation is valid for even a larger domain, we verify the condition in Proposition \ref{prop:extinction} for $q=1.$ We need to compute
\eqstar{
\tilde{p}_0 := \inf_{x\in\cO} {\left\{\E\bigl[e^{-d \eta^x}\bigr]\right\}} ~~ \text{and} ~~ 
\tilde{p}_3 := 1-\tilde{p}_0.
}
In the absence of an explicit formula for the Laplace transform of exit time of Brownian motion from a multidimensional rectangular domain, we estimate it by sampling the exit time using the numerical library developed by Lejay \cite{lejay:inria-00561409}. The integrability constant needs to be less than the threshold $\gamma$ in Proposition~\ref{prop:extinction} for our probabilistic representation to be valid. 
In view of~\eqref{eq:limit threshold}, the choice of threshold is given as $\gamma = \sqrt[3]{\tfrac{4}{27\tilde{p}_0^2 (1-\tilde{p}_0)}}.$ From our estimation procedures, we deduce that $C_0 < \gamma$ for $r < 0.484.$ Here, we exhibit our numerical results for $5 \times 10^5$ Monte Carlo sample paths with $r=0.48$ and different values of $x$ in Table~\ref{tab:example2D}. We observe that our estimator remains accurate but with an increased variance which is an expected effect due to an increased variance in the estimation of exit time and position of multi-dimensional Brownian motion using the library \texttt{exitbm}.
\begin{table}[H]
\centering
\small
\begin{tabular}{c c c c c c}
\toprule
$x$ & Estimate & 99\% conf. interval & Std. Dev./Mean & Relative error & Run time (secs)\\
\toprule
(0.0,0.0) & -0.0001 & [-0.0021, 0.0019] & -- & -- & 30\\
(0.1,0.0) & 0.0998 & [0.0978, 0.1018] & 5.4883 & 0.5726\% & 34\\
(0.2,0.1) & 0.3095 & [0.3075, 0.3114] & 1.7342 & 0.0426\% & 30\\
(0.2,0.2) & 0.4245 & [0.4226, 0.4265] & 1.2646 & 0.4144\% & 25\\
\hline
\end{tabular}
\caption{{\em Numerical results for the unbiased estimator of probabilistic representation of semi-linear PDE \eqref{eq:example3 semi-linear} in $d=2$ with $\beta = 2.$}}
\label{tab:example2D}
\end{table}

Typically, finite difference methods are quite commonly used to numerically solve PDEs. However, such methods are practically implementable and provide stable estimates only when $d \leq 3.$ Thus, we study PDE \eqref{eq:example3 semi-linear} in the case of $d=4$ to illustrate the broader applicability of our method. We extend the modules in library \texttt{exitbm} for our purpose using the existing functions. For $r < \tfrac{\pi}{16},$ the integrability constant $C_0 < 1$ and thus the estimator will have finite moments. Furthermore, through numerical estimation procedures, we have that $C_0 < \gamma$ for $r < 0.242.$ In Table~\ref{tab:example4D}, we exhibit our results with $5 \times 10^5$ Monte Carlo sample paths for $r=0.24$ which shows that our probabilistic representation provides an accurate estimator of the true solution.
\begin{table}[H]
\centering
\small
\begin{tabular}{c c c c c c}
\toprule
$x$ & Estimate & 99\% conf. interval & Std. Dev./Mean & Relative error & Run time (secs)\\
\toprule
(0,0,0,0) & 0.0000 & [-0.0010, 0.0012] & -- & -- & 190\\
(0.1,0,0,0) & 0.1009 & [0.0999, 0.1019] & 2.6185 & 0.5584\% & 309\\
(0.1,0.1,0,0) & 0.2021 & [0.2012, 0.2031] & 1.2986 & 0.2924\% & 413\\
(0.1,0.1,0.1,0) & 0.3097 & [0.3087, 0.31072] & 0.8750 & 0.1305\% & 510\\
\hline
\end{tabular}
\caption{{\em Numerical results for the unbiased estimator of probabilistic representation of semi-linear PDE \eqref{eq:example3 semi-linear} in $d=4$ with $\beta = 4.$}}
\label{tab:example4D}
\end{table}

\section{Conclusions}
In this work, we introduced a probabilistic representation for the solution of semi-linear elliptic PDEs with polynomial non-linearity by using the theory of branching diffusion processes. We performed a detailed analysis to derive explicit conditions under which our representations remain valid. In the linear gradient case, we essentially established that it holds provided that the domain is small enough. In the general case, we need to assume further that the nonlinear gradient term vanishes at the boundary of the domain to balance the explosion of the weight functions. As illustrated by our analysis of the one-dimensional case, this feature is inherent to any automatic differentiation formula. This restricts the choice of semi-linear elliptic PDEs for which we could use our probabilistic representation to obtain numerical estimates. For semi-linear elliptic PDEs of  form \eqref{eq:semi-linearpde}, we illustrated the applicability of our theoretical results to obtain numerical estimates with the help of several examples including multi-dimensional. Finally, in the numerical implementation of semi-linear elliptic PDEs driven by general diffusion processes, the error analysis due to the discretization of the process has been left for future research.

\section*{Acknowledgement}
The authors would like to thank Pierre Henry-Labord\`ere, Zhenjie Ren, Xiaolu Tan and Nizar Touzi for their valuable comments and suggestions. The first author research was conducted while at CMAP, \'Ecole Polytechnique and is part of the Chair {\it Financial Risks} of the {\it Risk Foundation}. The second author acknowledges the financial support of ERC 321111 Rofirm.

\begin{appendix}
\section{Proof of Proposition~\ref{prop:autodiff}}
\label{proof_autodiff}

For the purpose of clarity, we split the proof in two lemmas.

\begin{lemma}
\label{lem:adfin}
For any $g$ bounded measurable, the map  $x\mapsto\E[\int_0^{{\eta^x}} {e^{-\beta s} g(\bX^{x}_{s}) \,ds}]$ belongs to $\cC^1(\cO)\cap\cC(\bar{\cO})$ and 
\eqstar{ 
D \E\left[ \int_0^{{\eta^x}} {e^{-\beta s} g(\bX^{x}_{s}) \,ds} \right] &=  \E\left[ \int_0^{{\eta^x}} {e^{-\beta s}g(\bX^{x}_{s}) \cW_{\cO}(s,x,W) \,ds} \right].
}
\end{lemma}

\begin{proof}
The proof consists of four steps. The first three parts essentially follow by repeating the arguments in Gobet~\cite{gobet2004revisiting}. 
Denote
\eqstar{
\chi(s,x):=\E\left[g\left(\bX_{s}^{x}\right)\ind_{s< {\eta^x}}\right].
}

\vspace{0.5pc}

\noindent\textsl{First step.} 
Let us collect first preliminary results on the regularity of the function $\chi$.
Recall that, as established in the proof of Lemma~\ref{lem:continuity}, we have the following representation
\eqstar{
\chi(s,x) = \int_{\cO} {G(s,x;0,y) g(y) \,dy},
}
where $G$ is the Green function of PDE~\eqref{eq:parabolic} (see, \textit{e.g.}, Lady\v{z}enskaja \textit{et al.}~\cite[Theorem 4.16.2]{ladyzenskaja1968parabolic} or Friedman~\cite[Theorem 3.16]{friedman64parabolic}). 
Furthermore, in view of Theorem~4.16.3 in~\cite{ladyzenskaja1968parabolic}, for any $0<t<T$, we can differentiate under the integral sign and thus $\chi$ belongs to $\cC^{1,2}_b([t,T]\x\cO)$ and satisfies
\eqlnostar{eq:linearpde2}{
\partial_t \chi - \cL \chi  =0,~~ \text{in } [t,T]\x\cO.
}
In addition, under Assumption~\ref{ass:autodiff}, $\chi$ belongs to $C^{1,3}([t,T]\x\cO)$ (see, \textit{e.g.}, Friedman~\cite[Theorem 3.10]{friedman64parabolic}).

\vspace{0.5pc}

\noindent\textsl{Second step.} Next we show that $N_\theta:= D \chi(s-\theta,\bX^x_\theta) Y^x_\theta$ is a martingale up to time $\zeta_s$, \textit{i.e.}, $(N_{\zeta_s\wedge \theta})_{\theta\geq0}$ is a martingale.
Using It\^o's formula, we obtain
\eqstar{
N_{\zeta_s\wedge \theta} = D \chi(s,x) + \sum^{d}_{i=1} \int_0^{\zeta_s\wedge \theta}  \left( D \chi(s-r,\bX^x_r) D \sigma_i(\bX^x_r) + \sigma_i^{\top}(\bX^x_r) D^2 \chi(s-r,\bX^x_r)\right)Y^x_r \, dW^i_r.
}
In addition, due to Lemma~\ref{lem:delarue}, there exists $t<s$ such that $\zeta_s\leq t$. As the maps $D \chi$ and $D^2 \chi$ are bounded in $[s-t,s]\x\cO$, we conclude that $(N_{\zeta_s\wedge \theta})_{\theta\geq 0}$ is a martingale.

\vspace{0.5pc}

\noindent\textsl{Third step.} Then let us show that 
\eqstar{
D \chi(s,x) = \E\Bigl[g(\bX^x_{s})\ind_{s<{\eta^x}} \cW_{\cO}\left(s,x,W\right) \Bigr].
}
Indeed, It\^o's formula yields that
\eqstar{
\chi(s-\zeta_s,\bX^x_{\zeta_s})
&=\chi(s,x) + \int_0^{\zeta_s}D \chi(s-r,\bX_r^x) \sigma(\bX_r^x) \, dW_r.
}
Successively using Markov property and It\^o's isometry, we obtain
\eqstar{
& \E\left[g\left(\bX_s^{x}\right)\ind_{s\leq {\eta^x}}\left( \int_0^{\zeta_s}{\theta_s(r,\bX_s^r) \left(\sigma^{-1}(\bX_r^x) Y_r^x\right)^{\top}\, dW_r}\right)^{\top}\right] \\
	& = \E\left[\chi\left(s-\zeta_s,\bX^x_{\zeta_s}\right)\left( \int_0^{\zeta_s}{\theta_s(r,\bX_s^r) \left(\sigma^{-1}(\bX_r^x) Y_r^x\right)^{\top}\, dW_r}\right)^{\top}\right] \\
    & = \E\left[\int_0^{\zeta_s}{\theta_s(r,\bX_r^x) N_r\, dr}\right].
}
Once more, successively using the results that $N_{\zeta_s\wedge r}=\E[N_{\zeta_s}\,|\,\cF_r]$ and $\int_0^{\zeta_s}{\theta_s(r,\bX_r^x)\, dr} =1$, we deduce that
\eqstar{
\E\left[\int_0^{\zeta_s}{\theta_s(r,\bX_r^x) N_r\, dr}\right] 
= \E\left[N_{\zeta_s} \int_0^{\zeta_s}{\theta_s(r,\bX_r^x)\, dr}\right]
= \E\left[N_{\zeta_s}\right]
=D \chi(s,x).
}

\noindent\textsl{Fourth step.} We are now in a position to complete the proof. From Fubini's theorem we obtain 
\eqlnostar{eq:interres1}{
\E \Bigl[\int^{{\eta^x}}_0 e^{-\beta s}g(\bX^x_{s}) \,ds\Bigr] = \int^{\infty}_0 e^{-\beta s} \E \Bigl[ g(\bX^x_s)\ind_{s<{\eta^x}} \Bigr] \,ds.
}
Finally, in view of the estimates in Proposition~\ref{prop:autodiff}, we can apply differentiation under the integral sign in~\eqref{eq:interres1}. This completes the proof.
\end{proof}

\begin{lemma}
The map  $x\mapsto\E[e^{-\beta {\eta^x}} h(\bX^{x}_{{\eta^x}})]$  belongs to $\cC^1(\cO)\cap\cC(\bar{\cO})$ and 
\eqstar{ 
D \E\left[ e^{-\beta {\eta^x}} h(\bX^{x}_{{\eta^x}}) \right] &=  \E\left[ e^{-\beta {\eta^x}} h(\bX^{x}_{{\eta^x}}) \cW_{\del\cO}(x,W) \right].
}
\end{lemma}

\begin{proof}
This proof is similar to the proof of Lemma~\ref{lem:adfin} and we split it in three parts. Denote
\eqstar{
\varphi(x):=\E\left[e^{-\beta {\eta^x}} h\left(\bX^{x}_{{\eta^x}}\right)\right].
}

\vspace{0.5pc}

\noindent\textsl{First step.} Let us collect first preliminary results on the regularity of the function $\varphi$.
Recall that $\varphi\in\cC^2(\cO)\cap\cC(\bar{\cO})$ satisfies the following PDE:
\eqlnostar{eq:linearpde3}{
\cL \varphi - \beta \varphi  =0 ~~ \text{in } \cO,\quad~ \varphi  = h ~~ \text{on } \del\cO.
}
In addition, since $h\in\cC^{1,\alpha}(\bar{\cO})$ by assumption, it is known that $\varphi\in\cC^{1,\alpha}(\bar{\cO})$ (see, \textit{e.g.}, Gilbarg and Trudinger~\cite[Theorem 8.34]{gilbarg2015elliptic}). Furthermore, under Assumption~\ref{ass:autodiff}, the function $\varphi$ belongs to $C^{3}(\cO)$ (see, \textit{e.g.}, Gilbarg and Trudinger~\cite[Theorem 6.17]{gilbarg2015elliptic}).

\vspace{0.5pc}

\noindent\textsl{Second step.} Next we show that $N_s:=e^{-\beta s} D \varphi(\bX^x_s) Y^x_s$ is a martingale up to time $\eta^x$, \textit{i.e.}, $(N_{{\eta^x}\wedge s})_{s\geq 0}$ is a martingale. Using It\^o's formula, we obtain
\eqstar{
N_{{\eta^x}\wedge s} = D \varphi(x) + \sum^d_{i=1}\int_0^{{\eta^x}\wedge s} e^{-\beta r}\left(D \varphi(\bX^x_r)D \sigma_i(\bX^x_r) + \sigma^{\top}_i(\bX^x_r) D^2 \varphi(\bX^x_r)\right)Y^x_r \, dW^i_r.
}
Thus  $(N_{{\eta^x}\wedge s})_{s\geq 0}$ is a local martingale. To conclude, it remains to observe that for any $s\geq 0$,
\eqstar{
\E\left[\sup_{0\leq r\leq s} \left\{N_{{\eta^x}\wedge r}\right\}\right]<\infty,
}
since $D \varphi$ is bounded
and $\sup_{0\leq r\leq s} \{|Y^x_r|\}$ admits finite moments.

\vspace{0.5pc}

\noindent\textsl{Third step.} We are now in a position to complete the proof. Indeed, It\^o's formula yields for all $s\geq 0$,
\eqstar{
h(\bX^x_{\eta^x}) = \varphi(x) +  \int_0^{{\eta^x}}  D \varphi(\bX_r^x) \sigma (\bX_r^x) \, dW_r.
}
Using It\^o's isometry, we obtain
\eqstar{
\E\left[h\left(\bX_{\eta^x}^{x}\right)\left( \int_0^{\zeta_T}{\theta_T(r,\bX_s^r) \left(\sigma^{-1}(\bX_r^x) Y_r^x\right)^\top \, dW_r}\right)^\top\right] 
     = \E\left[\int_0^{\zeta_T}{\theta_T(r,\bX_r^x) N_r\, dr}\right].
}
Finally, successively using the results that $N_{\zeta_T\wedge r}=\E[N_{\zeta_T}\,|\,\cF_r]$ and $\int_0^{\zeta_T}{\theta_T(r,\bX_r^x)\, dr} =1$, we deduce that
\eqstar{
\E\left[\int_0^{\zeta_T}{\theta_T(r,\bX_r^x) N_r\, dr}\right] 
= \E\left[N_{\zeta_T} \int_0^{\zeta_T}{\theta_T(r,\bX_r^x)\, dr}\right]
= \E\left[N_{\zeta_T}\right]
= D \varphi(x).
}
\end{proof}

\end{appendix}

\bibliographystyle{abbrvnat}
\bibliography{nonlinellipticbranching}
\end{document}